\documentclass[12pt]{amsart}  

\usepackage{custom_preamble} 

\begin{document}

\title{On certain other sets of integers}

\author{Tom Sanders}
\address{Department of Pure Mathematics and Mathematical Statistics\\
University of Cambridge\\
Wilberforce Road\\
Cambridge CB3 0WB\\
England } \email{t.sanders@dpmms.cam.ac.uk}

\begin{abstract}
We show that if $A \subset \{1,\dots,N\}$ contains no non-trivial three-term arithmetic progressions then $|A|=O(N/\log^{3/4-o(1)}N)$.
\end{abstract}

\maketitle

\section{Introduction}

In this paper we refine the quantitative aspects of Roth's theorem on sets avoiding three-term arithmetic progressions.  Our main result is the following.
\begin{theorem}\label{thm.roth}
Suppose that $A \subset \{1,\dots,N\}$ contains no non-trivial three-term arithmetic progressions.  Then
\begin{equation*}
|A| = O(N / \log^{3/4 - o(1)} N).
\end{equation*}
\end{theorem}
The first non-trivial upper bound on the size of such sets was given by Roth \cite{rot::,rot::0}, and there then followed refinements by Heath-Brown \cite{hea::} and Szemer{\'e}di \cite{sze::2}, and later Bourgain \cite{bou::5,bou::1}, culminating in the above with $2/3$ in place of $3/4$.  Our interest is in bounds so we have not recorded references to the many proofs and generalisations of Roth's theorem not concerned with this issue.

Our argument follows \cite{bou::1} where the key new ingredient was a random sampling technique.  We couple this with the heuristic (see \cite{bou::2} for a rigorous formulation) that random sampling `increases dissociativity' to get a little bit more out of the method.  No progress is made on the algebraic analogue in $(\Z/3\Z)^n$.

Salem and Spencer \cite{salspe::} showed that the surface of high-dimensional convex bodies could be embedded in the integers to construct sets of size $N^{1-o(1)}$ containing no three-term progressions, and Behrend \cite{beh::} noticed that spheres are a particularly good choice.  Recently Elkin \cite{elk::} tweaked this further by thickening the spheres to produce the largest known progression-free sets, and his argument was then considerably simplified by Green and Wolf in the very short and readable note \cite{grewol::}.

In the next section we establish our basic notation before we present a sketch of the main argument in \S\ref{sec.sketch}.  The remainder of the paper then makes the sketch precise. 

\section{Basic notation: the Fourier transform and Bohr sets}

Suppose that $G$ is a finite abelian group.  We write $M(G)$ for the space of measures on $G$ and given two measures $\mu,\nu \in M(G)$ define their convolution $\mu \ast \nu$ to be the measure induced by
\begin{equation*}
C(G) \rightarrow C(G); f \mapsto \int{f(x+y)d\mu(x)d\nu(y)}.
\end{equation*}
There is a privileged measure $\mu_G \in M(G)$ on $G$ called Haar probability measure, and defined to be the measure assigning mass $|G|^{-1}$ to each element of $G$.

The significance of this measure is that it is the unique (up to scaling) translation invariant measure on $G$: for $x \in G$ and $\mu \in M(G)$ we define $\tau_x(\mu)$ to be the measure induced by
\begin{equation*}
C(G) \rightarrow C(G); f \mapsto \int{f(y)d\mu(y+x)},
\end{equation*}
and it is easy to see that $\tau_x(\mu_G)=\mu_G$ for all $x \in G$.

We use Haar measure to pass between the notion of function $f\in L^1(\mu_G)$ and measure $\mu \in M(G)$.  Indeed, since $G$ is finite we shall often identify $\mu$ with $d\mu/d\mu_G$, the Radon-Nikodym derivate of $\mu$ with respect to $\mu_G$.  In light of this we can extend the notion of translation and convolution to $L^1(\mu_G)$ with ease: given $f \in L^1(\mu_G)$ and $x \in G$ we define the translation of $f$ by $x$ point-wise by
\begin{equation*}
\tau_x(f)(y):=\frac{d(\tau_x(fd\mu_G))}{d\mu_G)}(y) = f(x+y) \textrm{ for all } y \in G,
\end{equation*}
and given $f,g \in L^1(\mu_G)$ we define convolution of $f$ and $g$ point-wise by
\begin{equation*}
f \ast g(x):=\frac{d((fd\mu_G) \ast (g d\mu_G))}{d\mu_G}(x) = \int{f(y)g(x-y)d\mu_G(y)},
\end{equation*}
and similarly for the convolution of $f \in L^1(\mu_G)$ with $\mu \in M(G)$.  

Convolution operators can be written in a particularly simple form with respect to the Fourier basis which we now recall.  We write $\wh{G}$ for the dual group, that is the finite Abelian group of homomorphisms $\gamma:G \rightarrow S^1$, where $S^1:=\{z \in \C:|z|=1\}$.  Given $\mu \in M(G)$ we define $\wh{\mu} \in \ell^\infty(\wh{G})$ by
\begin{equation*}
\wh{\mu}(\gamma):=\int{\overline{\gamma}d\mu} \textrm{ for all } \gamma \in \wh{G},
\end{equation*}
and extend this to $f \in L^1(G)$ by $\wh{f}:=\wh{fd\mu_G}$. It is easy to check that $\wh{\mu \ast \nu} = \wh{\mu}\cdot \wh{\nu}$ for all $\mu,\nu \in M(G)$ and $\wh{f \ast g} = \wh{f} \cdot \wh{g}$ for all $f,g \in L^1(\mu_G)$.

It is also useful to have an inverse for the Fourier transform: given $g \in \ell^1(\wh{G})$ we define $g^\vee \in C(G)$ by
\begin{equation*}
g^\vee(x)=\sum_{\gamma \in \wh{G}}{g(\gamma)\gamma(x)} \textrm{ for all } x \in G,
\end{equation*}
and note that $(g^\vee)^\wedge=g$.  It is now natural to define the convolution of two functions $f,g \in \ell^1(\wh{G})$ to be
\begin{equation*}
f \ast g(\gamma) = \sum_{\gamma' \in \wh{G}}{f(\gamma')g(\gamma - \gamma')} \textrm{ for all } \gamma \in \wh{G},
\end{equation*}
so that $(f \ast g)^\vee = f^\vee \cdot g^\vee$.

As is typical we shall establish the presence of non-trivial arithmetic progressions by counting them, and to this end we introduce the tri-linear form
\begin{equation*}
T(f,g,h):=\int{f(x-y)g(y)h(x+y)d\mu_G(x)d\mu_G(y)},
\end{equation*}
which has a rather useful Fourier formulation:
\begin{equation*}
T(f,g,h)=\sum_{\gamma \in \wh{G}}{\wh{f}(\gamma)\wh{g}(-2\gamma)\wh{h}(\gamma)}.
\end{equation*}

Following \cite{bou::1} we use a slight generalization of the traditional notion of Bohr set, letting the width parameter vary according to the character.  The advantage of this definition is that the meet of two Bohr sets in the lattice of Bohr sets is then just their intersection.

A set $B$ is called a \emph{Bohr set} if there is a \emph{frequency set} $\Gamma$ of characters on $G$, and a \emph{width function} $\delta \in (0,2]^\Gamma$ such that
\begin{equation*}
B=\{x \in G: |1-\gamma(x)| \leq \delta_\gamma \textrm{ for
all }\gamma \in \Gamma\}.
\end{equation*}
The size of the set $\Gamma$ is called the \emph{rank} of $B$ and is denoted $\rk(B)$.  Technically the same Bohr set can be defined by different frequency sets and width functions; we make the standard abuse that when we introduce a Bohr set we are implicitly fixing a frequency set and width function.

There is a natural way of dilating Bohr sets which will be of particular use to us.  Given such a $B$, and $\rho \in \R^+$ we shall write $B_\rho$ for the Bohr set with frequency set $\Gamma$ and width function\footnote{Technically width function $\gamma \mapsto \min\{\rho \delta_\gamma,2\}$.} $\rho\delta$ so that, in particular, $B=B_1$ and more generally $(B_\rho)_{\rho'} =B_{\rho\rho'}$.   

With these dilates we say that a Bohr set $B'$ is a \emph{sub-Bohr set} of another Bohr set $B$, and write $B' \leq B$, if 
\begin{equation*}
B'_\rho \subset B_{\rho} \textrm{ for all } \rho \in \R^+.
\end{equation*}
Finally, we write $\beta_\rho$ for the probability measure induced on $B_\rho$ by $\mu_G$, and $\beta$ for $\beta_1$. 

Throughout the paper $C$s will denote absolute, effective, but unspecified constants of size greater than $1$ and $c$s will denote the same of size at most $1$.  Typically the constants will be subscripted according to the result which they come from and superscripted within arguments.

We shall also find a variant of the big-$O$ notation useful.  Specifically, if $X,Y \geq 1$ then we write $X=\widetilde{O}(Y)$ to mean $X = O(Y\log^{O(1)}2Y)$, and if $X,Y \leq 1$ then $X=\widetilde{\Omega}(Y)$ to mean $X^{-1} = \widetilde{O}(Y^{-1})$.

\section{Sketching the argument}\label{sec.sketch}

The following discussion is based on the approaches to Roth's theorem developed by Bourgain in \cite{bou::5} and then \cite{bou::1}.  There have been a number of expositions of the first of these papers -- the interested reader may wish to consult Tao and Vu \cite{taovu::} -- and the second paper contains quite a detailed heuristic discussion of the argument. 

It should be remarked that the present section is not logically necessary for the rest of the paper, and the reader who is not already familiar with the basic ideas may find it difficult to follow.

We shall concentrate on the problem of showing that if $G$ is a finite abelian group of odd order and $A \subset G$ has density $\alpha \in (0,1]$ then
\begin{equation*}
T(1_A,1_A,1_A) =\exp(-\widetilde{O}(\alpha^{-C})),
\end{equation*}
ideally with as small a constant $C$ as possible.

It seems unlikely that the odd order condition is necessary in light of the work of Lev \cite{lev::}, but removing it seems to add technicalities which obscure the underlying ideas.  In any case the passage from finite abelian groups of odd order to $\{1,\dots,N\}$ is unaffected and follows from the usual Fre{\u\i}man embedding.

To begin recall that by the usual application of the inversion formula and the triangle inequality, if $A$ has somewhat fewer than expected three-term progressions then
\begin{equation}\label{eqn.cu}
\sum_{0_{\wh{G}} \neq \gamma \in \Spec_{\Omega(\alpha)}(1_A)}{|\wh{1_A}(\gamma)|^2|\wh{1_A}(-2\gamma)|}=\Omega(\alpha^3).
\end{equation}

\subsection{The basic approach}\label{ss.bas}  Roth \cite{rot::0} noted that the above implies, in particular, that the large spectrum is non-empty, and so we have a non-trivial character $\gamma$ with $|\wh{1_A}(\gamma)|=\Omega(\alpha^2)$.  This can be used to give a density increment on (a translate of) $I:=\{x \in G: \gamma(x) \approx 1\}$ of the form $\alpha \mapsto \alpha(1+\Omega(\alpha))$.

In \cite{bou::5} Bourgain noted that approximate level sets of characters -- sets like $I$ -- behave roughly like groups, and that crucially the preceding argument can be run relative to anything which behaves in that way.

The intersection of many approximate level sets is a Bohr set (hence their importance to us) and working through the previous for a set $A$ of relative density $\alpha$ on a $d$-dimensional Bohr set $B$ gives a new Bohr set $B'$ with
\begin{equation}\label{eqn.bq}
d \mapsto d+1\textrm{ and } \alpha \mapsto \alpha(1+\Omega(\alpha)).
\end{equation}
In general we only have a bound of the form
\begin{equation}\label{eqn.bf}
\mu_G(B') \geq \exp(-\widetilde{O}(d))\mu_G(B),
\end{equation}
although for groups of bounded exponent this can be much improved.  Indeed, if $G$ is the model group $(\Z/3\Z)^n$ then $\mu_G(B') = \Omega(\mu_G(B))$, and this is the reason for the difference in the bounds in the model setting (see \cite{mes::}) and here.

A density increment of the form in (\ref{eqn.bq}) cannot proceed for more that $\widetilde{O}(\alpha^{-1})$ steps, and so there must be some point where $A$ has about the expected number of three-term progressions on a Bohr set $B'$ of dimension $d+\widetilde{O}(\alpha^{-1})$.  It follows that
\begin{equation*}
\mu_G(B') \geq \exp(-\widetilde{O}(d))\dots \exp(-\widetilde{O}(d+\alpha^{-1}))\mu_G(B)
\end{equation*}
which tells us that
\begin{equation}\label{eqn.bgbas}
T(1_A,1_A,1_A) = \exp(-\widetilde{O}(d\alpha^{-1}+\alpha^{-2}))\mu_G(B)^2.
\end{equation}
We are interested in the case $d=0$ and $B=G$ when this gives the bound in \cite{bou::5}.

\subsection{The energy increment approach}\label{ss.eia}  

The previous argument only uses the fact that (\ref{eqn.cu}) tells us the large spectrum is non-empty.  It is natural to try to get a density increment by noting that the Bohr set $B$ with frequency set $\Spec_{\Omega(\alpha)}(1_A)$ majorises that spectrum, and hence $1_A \ast \beta$ has large $L^2(\mu_G)$-mass. 

An  energy increment technique motivated by the work of Heath-Brown \cite{hea::} and Szemer{\'e}di \cite{sze::2} (essentially Lemma \ref{lem.basicl2}) now tells us that $A$ has density $\alpha(1+\Omega(1))$ on (a translate of) $B$.  It remains to estimate the dimension of $B$.  This is certainly at most the size of the largest dissociated subset of the large spectrum, and this can be bounded by Chang's theorem \cite{cha::0}; in this case it is $\widetilde{O}(\alpha^{-2})$.

All of this can, itself, be done for a set $A$ of relative density $\alpha$ in a Bohr set $B$ of dimension $d$.  In which case either we have many three-term progressions or else we have a new Bohr set with
\begin{equation*}
d \mapsto d+\widetilde{O}(\alpha^{-2}) \textrm{ and } \alpha \mapsto \alpha(1+\Omega(1)).
\end{equation*}
This process can be iterated at most $O(\log 2\alpha^{-1})$ times and so proceeding as before with the lower bound (\ref{eqn.bf}) we conclude that
\begin{equation}\label{eqn.bgnrg}
T(1_A,1_A,1_A) = \exp(-\widetilde{O}(d+\alpha^{-2}))\mu_G(B)^2.
\end{equation}
This represents no improvement over (\ref{eqn.bgbas}) in our case of interest, but for large values of $d$ it does and we shall find this fact useful later on.

\subsection{Comparing the approaches}  In the approach in \S\ref{ss.bas} we only looked at $\widetilde{O}(\alpha^{-1})$ different characters but we did so one at a time and every time we iterated we used (\ref{eqn.bf}) which cost us a factor exponential in the dimension of the Bohr set.  In the approach in \S\ref{ss.eia} we looked at $\widetilde{O}(\alpha^{-2})$ characters but took almost all of them together meaning that we only lost the factor from (\ref{eqn.bf}) a very small number of times.

To get a better bound one should like to improve Chang's theorem, but that is not possible without additional information about $A$.  (See \cite{gre::5} for details.)  However, suppose that we are in the extreme case of Chang's theorem and there are many dissociated characters in the large spectrum.  They each lead to a different density increment of the type in \S\ref{ss.bas} and one might hope that these increments are in some sense `independent' (since the characters are) so that rather than having to do them one after the other and suffering the expense of repeated application of (\ref{eqn.bf}), we could do them all at once and only suffer the cost once.  It is this hope to which we now turn our attention following \cite{bou::1}.

\subsection{The random sampling strategy}  In the first instance the argument proceeds by removing the mass in (\ref{eqn.cu}) corresponding to $\Spec_{\sqrt{\alpha}}(A)$ using the Heath-Brown-Szemer{\'e}di energy increment technique as in \S\ref{ss.eia}.  In this case we get a Bohr set of dimension $\widetilde{O}(\alpha^{-1})$ on which we have a density increment of the form $\alpha \mapsto \alpha(1+\Omega(1))$. Otherwise, by dyadic decomposition, there is some $\tau \in [\Omega(\alpha),\sqrt{\alpha}]$ such that 
\begin{equation}\label{eqn.iwq}
\sum_{\gamma \in S_\tau}{|\wh{1_A}(\gamma)|^2} =\widetilde{\Omega}(\tau^{-1}\alpha^2),
\end{equation}
where $S_\tau:=\Spec_{2\tau}(1_A) \setminus \Spec_\tau(1_A)$.  We shall split into two cases depending on whether $S_\tau$ contains a large dissociated set or not.

Suppose that $T$ is a dissociated subset of $S_\tau$, and let $c_\lambda$ be the phase of $\wh{1_A}(\lambda)$.  By expanding out we have
\begin{equation}\label{eqn.expnd}
\langle 1_A,\prod_{\lambda\in T}{(1+c_\lambda\Re\lambda)}\rangle_{L^2(\mu_G)} = \alpha + \sum_{\lambda \in \Lambda}{|\wh{1_A}(\lambda)|} + \dots.
\end{equation}
If the remaining terms really were an error term then we would have that the inner product is at least $\alpha(1+\Omega(\tau |T|))$.

Now, a Riesz product $p$ generated by a set $\Lambda$ of characters is essentially invariant under convolution with $\beta$, where $B$ is the Bohr set with frequency set $\Lambda$.  Thus
\begin{equation*}
\langle 1_A,p \rangle_{L^2(\mu_G)} \approx \langle 1_A \ast \beta,p\rangle_{L^2(\mu_G)} \leq \|1_A \ast \beta\|_{L^\infty(\mu_G)} \int{pd\mu_G}.
\end{equation*}
Of course, $1_A \ast \beta(x)$ is just the density of $A$ on $x+B$ and so if the inner product on the left is large and the integral of the Riesz product is $1$ then we have a density increment.

Since $T$ is dissociated, the integral of our Riesz product is $1$ and we would get the density increment $\alpha \mapsto \alpha(1+\Omega(\tau |T|))$ on a Bohr set of dimension $O(|T|)$ if the remaining terms in (\ref{eqn.expnd}) were error terms.  Of course, they are not and to deal with this we sample from $T$.  Let $\Lambda$ be a subset of $T$ with $\lambda \in T$ chosen independently with probability $\theta$.  Then
\begin{eqnarray}\nonumber 
\E\langle 1_A, \prod_{\lambda\in \Lambda}{(1+c_\lambda\Re\lambda)}\rangle_{L^2(\mu_G)}&=&\langle 1_A,\prod_{\lambda \in T}(1+\theta c_\lambda \Re \lambda)\rangle_{L^2(\mu_G)}\\ \nonumber &\geq &\langle 1_A,\exp(\sum_{\lambda \in T}{\theta c_\lambda \Re\lambda})\rangle_{L^2(\mu_G)}(1+O(\theta^2|T|))\\ \nonumber & \geq & \langle 1_A,1+\sum_{\lambda \in T}{\theta c_\lambda \Re\lambda}\rangle_{L^2(\mu_G)}(1+O(\theta^2|T|))\\ \label{eqn.gtinc} & \geq & \alpha +\theta |T|\tau \alpha/2
\end{eqnarray}
if $\theta \leq c\tau$.  Now we get a Bohr set of (expected) dimension $O(\theta |T|)$ on which we have density $\alpha(1+\Omega(\tau \theta |T|))$.

This can all be done for a set $A$ of relative density $\alpha$ in a $d$-dimensional Bohr set, and taking $\theta = c\tau$ we have a dichotomy: if there are at least $k$ dissociated elements in $S_\tau$ then we get a new Bohr set with
\begin{equation}\label{eqn.j}
d \mapsto d+O(\tau k) \textrm{ and } \alpha \mapsto \alpha(1+\Omega(\tau^2k));
\end{equation}
if there are at most $k$ dissociated elements in $S_\tau$ then we apply the Heath-Brown-Szemer{\'e}di energy increment technique to get a Bohr set of dimension $d+O(k)$ on which we have density $\widetilde{\Omega}(\tau^{-1}\alpha)$ by equation (\ref{eqn.iwq}).  In this instance we apply (\ref{eqn.bgnrg}) and are done.  

Now, (\ref{eqn.j}) can happen at most $O(\tau^{-2}k^{-1})$ times which means that we have at most that many applications of (\ref{eqn.bf}), while the dimension of the Bohr set is at most
\begin{equation*}
O(\tau^{-2}k^{-1}).O(\tau k) + \widetilde{O}(\alpha^{-1})=\widetilde{O}(\alpha^{-1})
\end{equation*}
since $\tau = \Omega(\alpha)$.   After that the iteration must terminate either with many progressions or an application of (\ref{eqn.bgnrg}) and we find that for $A \subset G$ of density $\alpha$ we have
\begin{equation}\label{eqn.cd}
T(1_A,1_A,1_A) \geq \exp(-\widetilde{O}(\tau^{-2}k^{-1}\alpha^{-1}+k + \tau^2\alpha^{-2})).
\end{equation}
Of course $\tau \in [\Omega(\alpha),\sqrt{\alpha}]$ so we can optimise with $k=\alpha^{-3/2}$ to get the main result in \cite{bou::1}.

\subsection{Relaxing the dissociation}  To get the density increment from (\ref{eqn.gtinc}) we used the fact that $T$ is dissociated so that the Riesz product has integral $1$, but if the integral is at most $1+c\theta |T|\tau$ then the argument we used still works.  Moreover, the random sampling process actually makes the integral closer to $1$, and it is this which we shall exploit to improve (\ref{eqn.cd}).

We start with a decomposition of Bourgain from \cite{bou::4} the existence of which is an easy induction: for a parameter $m \in \N$,
\begin{enumerate}
\item either half of $S_\tau$ has no dissociated set of size at least $m$;
\item or half of $S_\tau$ is a disjoint union of dissociated sets of size $m$.
\end{enumerate}
In the former case we get a Bohr set of dimension $m$ on which we have density $\Omega(\tau^{-1})=\Omega(\alpha^{1/2})$ by the Heath-Brown-Szemer{\'e}di energy increment technique, and we are done after applying the bound (\ref{eqn.bgnrg}).  In the latter case write $T$ for the union, and note by Rudin's inequality and the triangle inequality that
\begin{equation}\label{eqn.en}
1_{T} \ast \dots \ast 1_{T}(0_{\wh{G}}) \leq O(r/m)^{(r-2)/2}|T|^{r-1},
\end{equation}
where the convolution is $r$-fold.  This sort of expression and related consequences have been exploited extensively in the work of Shkredov \cite{shk::3,shk::5} and Shkredov and Yekhanin \cite{shkyek::}, and that work was one of the starting points for this paper. 

Let $\Lambda$ be a subset of $T$ with $\lambda \in T$ chosen independently with probability $\theta$. Now, suppose (for simplicity) that all the phases $c_\lambda$ from the previous subsection are $1$.  Then we have
\begin{eqnarray*}
\E{\int{ \prod_{\lambda\in \Lambda}{(1+c_\lambda\Re\lambda)}d\mu_G}}&  \leq & 1+\theta \int{\sum_{\lambda \in \Lambda}{c_\lambda\Re\lambda}d\mu_G} +  \sum_{r=2}^\infty{\frac{\theta^r}{r!}1_{T} \ast \dots \ast 1_{T}(0_{\wh{G}})}\\ & \leq & 1 + \sum_{r=2}^\infty{\frac{m}{r|T|}O\left(\frac{\theta |T|}{\sqrt{mr}}\right)^r}
\end{eqnarray*}
since none of the characters is trivial.  Of course if $|T| \geq C \tau^{-1}\sqrt{m}$ then we may take $\theta |T| = c' \sqrt{m}$ whilst keeping $\theta \leq c\tau$, and the upper bound is then just $1+O(\theta^2|T|)$.  Assuming that this typical behaviour actually happens we get a density increment of $\alpha \mapsto \alpha(1+\Omega(\tau \sqrt{m}))$ on a Bohr set of dimension $O(\sqrt{m})$.  For a suitable choice of $m$ this will be an improvement on what we had before.

In the non-relative situation it is easy to see that $T$ is as large as needed since it's at least half of $S_\tau$, which in turn has size $\Omega(\tau^{-3})$ from (\ref{eqn.iwq}).  However, when we relativise to Bohr sets this is not true any more.  To solve this, if $|T| \leq C \tau^{-1}\sqrt{m}$ then we shall arrange things so that we can take a subset $T'$ of size $\eta|T|$ such that
\begin{equation*}
\sum_{\lambda \in T'}{|\wh{1_A}(\gamma)|^2} =\widetilde{\Omega}(\eta\tau^{-1}\alpha^2),
\end{equation*}
and apply the usual Heath-Brown-Szemer{\'e}di energy increment technique to get a Bohr set of dimension $O(\eta \tau^{-1}\sqrt{m})$ on which we have density $\widetilde{\Omega}(\eta\tau^{-1}\alpha)$; we finish off this case when $|T|$ is small by applying (\ref{eqn.bgnrg}).

Relativising all this as in the previous subsection we get that our iteration procedure terminates inside $O(\tau^{-1}m^{-1/2})$ steps and
\begin{equation*}
T(1_A,1_A,1_A) \geq \exp(-\widetilde{O}(\tau^{-1}m^{-1/2}\alpha^{-1}+m + \eta \tau^{-1}\sqrt{m}+ \eta^{-2}\tau^2\alpha^{-2})).
\end{equation*}
Of course $\tau \in [\Omega(\alpha),\sqrt{\alpha}]$ so we can optimise with\footnote{There is some flexibility with $\eta$; $m$ is the critical parameter.} $\eta=\tau \alpha^{-1/2}$ and $m = \alpha^{-4/3}$ to get our main theorem. 

\section{The large spectrum}\label{sec.lspec}

In this section we shall be working with respect to a probability measure $\mu$ on $G$.  Later we shall take $\mu$ to be a sort of approximate Haar measure on approximate groups, but for the moment our analysis does not require this.

Our object of study is the spectrum of functions with respect to $\mu$: given a function $f \in L^1(\mu)$ and a parameter $\epsilon \in (0,1]$ we define the \emph{$\epsilon$-spectrum of $f$ w.r.t. $\mu$} to be the set
\begin{equation*}
\Spec_\epsilon(f,\mu):=\{\gamma \in \wh{G}: |(fd\mu)^\wedge(\gamma)| \geq \epsilon\|f\|_{L^1(\mu)}\}.
\end{equation*}
This definition extends the usual one from the case $\mu=\mu_G$.  In that case there is a rather simple estimate for the size of the spectrum which follows from Bessel's inequality, and it will be useful for us to have a relative version of this.

Defining notions by what they do, rather than what they are is a fairly standard idea, but it has seen particular success in the development of approximate analogues of concepts in additive combinatorics.  This is spelt out particularly clearly by Gowers and Wolf in \cite{gowwol::}, and the next definition (as well as similar later ones) is motivated by their discussion of quadratic rank.  We say that $\Lambda \subset \wh{G}$ is \emph{$K$-orthogonal w.r.t. $\mu$} if
\begin{equation*}
\int{|1+g^\vee|^2d\mu} \leq (1+K)(1+\|g\|_{\ell^2(\Lambda)}^2) \textrm{ for all }g\in \ell^2(\Lambda).
\end{equation*}
\begin{lemma}[Monotonicity of orthogonality]
Suppose that $\mu'$ is another probability measure, $\Lambda$ is $K$-orthogonal w.r.t. $\mu$, $\Lambda' \subset \Lambda$ and $K' \geq K$.  Then $\Lambda'$ is $K'$-orthogonal w.r.t. $\mu' \ast \mu$.
\end{lemma}
\begin{proof}
The only part worthy of remark is the replacing of $\mu$ by $\mu'\ast\mu$:  this is immediate on noting that $\tau_{-y}(1+g^\vee) = 1+h^\vee$ where $h(\lambda)=\overline{\lambda(y)}g(\lambda)$ and $\|h\|_{\ell^2(\Lambda)}^2=\|g\|_{\ell^2(\Lambda)}^2$.  It follows that $\Lambda$ is $K$-orthogonal w.r.t. $\tau_y(\mu)$, and we get the result on integrating.
\end{proof}
We define the \emph{$(K,\mu)$-relative size} of $\Lambda$ to be the size of the largest subset of $\Lambda$ that is $(K,\mu)$-orthogonal w.r.t. $\mu$.  By monotonicity the $(K',\mu' \ast \mu)$-relative size dominates the $(K,\mu)$-relative size; the $(0,\mu_G)$-relative size of $\Lambda$ is the size of the largest subset of $\Lambda$ not containing $\{0_{\wh{G}}\}$.

The Bessel bound mentioned earlier follows easily from this definition.
\begin{lemma}[The Bessel bound]\label{lem.apbess} Suppose that $f \in L^2(\mu)$ is not identically zero and write $L_f:=\|f\|_{L^2(\mu)}\|f\|_{L^1(\mu)}^{-1}$.  Then $\Spec_\epsilon(f,\mu)$ has $(1,\mu)$-relative size $O(\epsilon^{-2}L_f^{-2})$.
\end{lemma}
\begin{proof}
Suppose that $\Lambda \subset \Spec_{\epsilon}(f,\mu)$ is $(1,\mu)$-orthogonal.  Then
\begin{equation*}
1+\|g^\vee\|_{L^2(\mu)}^2 = \frac{1}{2}\left(\int{|1-g^\vee|^2d\mu} + \int{|1+g^\vee|^2d\mu} \right) \leq 2(1+\|g\|_{\ell^2(\Lambda)}^2)
\end{equation*}
for all $g \in \ell^2(\Lambda)$.  It follows (by letting $\|g\|_{\ell^2(\Lambda)}\rightarrow \infty$) that the map $\ell^2(\Lambda) \rightarrow L^2(\mu)$ has norm at most $2^{1/2}$ and so by duality the map $L^2(\mu) \rightarrow \ell^2(\Lambda)$ defined by $f \mapsto (fd\mu)^\wedge|_{\Lambda}$ has norm at most $2^{1/2}$, whence
\begin{equation*}
|\Lambda|\epsilon^2\|f\|_{L^1(\mu)}^2 \leq \sum_{\lambda \in \Lambda}{|(fd\mu)^\wedge(\lambda)|^2} \leq 2\|f\|_{L^2(\mu)}^2.
\end{equation*}
We get the result on rearranging.
\end{proof}
In the case $\mu=\mu_G$ this yields the usual Parseval bound which pervades work in the area.  The bound was first understood in the context of Bessel's inequality in the paper \cite{gretao::1} of Green and Tao.

Orthogonal random variables can be thought of as roughly pair-wise independent, and naturally we can improve the above bound if we insist on a greater level of independence.  The idea of doing this was developed by Chang in \cite{cha::0} following on from work of Bourgain \cite{bou::4}.  

Given a set of characters $\Lambda$ and a function $\omega:\Lambda \rightarrow D:=\{z \in \C: |z| \leq 1\}$ we define
\begin{equation*}
p_{\omega,\Lambda}:=\prod_{\lambda \in \Lambda}{(1+\Re (\omega(\lambda)\lambda))},
\end{equation*}
and call such a function a \emph{Riesz product for $\Lambda$}.  It is easy to see that all Riesz products are real non-negative functions.  They are at their most useful when they also have mass close to $1$: the set $\Lambda$ is said to be \emph{$K$-dissociated w.r.t. $\mu$} if
\begin{equation*}
\int{p_{\omega,\Lambda}d\mu} \leq \exp(K) \textrm{ for all } \omega:\Lambda \rightarrow D.
\end{equation*}
If $\mu$ is positive definite then one sees by Plancherel's theorem that the maximum above is attained for the constant function $1$ and so we have the following lemma.
\begin{lemma}[Dissociativity for positive definite measures]
Suppose that $\Lambda$ is a set of characters and $\mu$ is positive definite. Then $\Lambda$ is $\log \int{p_{1,\Lambda}d\mu}$-dissociated w.r.t. $\mu$.
\end{lemma}
As before we also have a monotonicity result.
\begin{lemma}[Monotonicity of dissociativity]
Suppose that $\mu'$ is another probability measure, $\Lambda$ is $K$-dissociated w.r.t. $\mu$, $\Lambda' \subset \Lambda$ and $K' \geq K$.  Then $\Lambda'$ is $K'$-dissociated w.r.t. $\mu' \ast \mu$.
\end{lemma}
The \emph{$(K,\mu)$-relative entropy} of a set $\Lambda$ is the size of the largest subset of $\Lambda$ that is $K$-dissociated w.r.t. $\mu$.  As before relative entropy is monotone; the $(0,\mu_G)$-relative entropy of a set is the size of its largest dissociated (in the usual sense) subset.

A slightly harder property of dissociated sets is that they satisfy a Chernoff-type estimate.  The proof of this that we give here essentially localises \cite[Proposition 3.4]{greruz::0}.
\begin{lemma}\label{lem.chernoff}
Suppose that $\Lambda$ is $K$-dissociated w.r.t. $\mu$.  Then
\begin{equation*}
\int{|\exp(g^\vee)|d\mu} \leq \exp(K+\|g\|_{\ell^2(\Lambda)}^2/2)\textrm{ for all }g \in \ell^2(\Lambda).
\end{equation*}
\end{lemma}
\begin{proof}
Suppose that $g \in \ell^2(\Lambda)$ and begin by noting that
\begin{equation*}
\int{|\exp(g^\vee)|d\mu} = \int{\exp(\Re g^\vee)d\mu} = \int{\prod_{\lambda \in \Lambda}{\exp(\Re (g(\lambda)\lambda))}d\mu}.
\end{equation*}
Now we have the elementary inequality $\exp(ty) \leq \cosh t + y \sinh t$ whenever $t \in \R$ and $-1 \leq y \leq 1$, so
\begin{equation*}
\int{|\exp(g^\vee)|d\mu} \leq \int{\prod_{\lambda \in \Lambda: g(\lambda) \neq 0}{\left(\cosh |g(\lambda)| + \frac{\Re(g(\lambda)\lambda)}{|g(\lambda)|}\sinh |g(\lambda)|\right)}d\mu},
\end{equation*}
with the usual convention that $t^{-1}\sinh t$ is $1$ if $t=0$.  We define $\omega \in \ell^\infty(\Lambda)$ by
\begin{equation*}
\omega(\lambda):=\frac{g(\lambda)\sinh |g(\lambda)|}{|g(\lambda)|\cosh |g(\lambda)|}
\end{equation*}
so that we certainly have $\|\omega\|_{\ell^\infty(\Lambda)} \leq 1$, and hence $p_{\omega,\Lambda}$ is a Riesz product and
\begin{equation*}
\int{|\exp(g^\vee)|d\mu} \leq \prod_{\lambda \in \Lambda}{\cosh |g(\lambda)|}\int{p_{\omega,\Lambda}d\mu}.
\end{equation*}
The result follows since $\cosh x \leq \exp(x^2/2)$.
\end{proof}
It is, perhaps, instructive to compare the conclusion of this lemma with the definition of being $K$-orthogonal w.r.t. $\mu$.

Green and Ruzsa used the above estimate (for $K=0$) in \cite{greruz::0} to prove Chang's theorem for sets, but our argument will proceed along the more traditional lines of establishing Rudin's inequality (w.r.t. to $\mu$) and then applying duality so that it applies to general functions.
\begin{lemma}[The Chang bound]\label{lem.changbd}  Suppose that $f \in L^2(\mu)$ is not identically zero and write $L_f:=\|f\|_{L^2(\mu)}\|f\|_{L^1(\mu)}^{-1}$.  Then $\Spec_\epsilon(f,\mu)$ has $(1,\mu)$-relative entropy $O( \epsilon^{-2}\log 2L_f)$.
\end{lemma}
\begin{proof}
Suppose that $\Lambda \subset \Spec_\epsilon(f,\mu)$ is $K$-dissociated w.r.t. $\mu$.  First we establish a relative version of Rudin's inequality.  Suppose that $k \in \N$ and $g \in \ell^2(\Lambda)$ has $\|g\|_{\ell^2(\Lambda)} = 2\sqrt{k}$.  Then Lemma \ref{lem.chernoff} tells us that
\begin{equation*}
\|\Re g^\vee\|_{L^{2k}(\mu)} \leq (2k)!^{1/2k}\exp(1/2k+4k/4k) = O(\sqrt{k}\|g\|_{\ell^2(\Lambda)}).
\end{equation*}
However, since $\Re( (-i)g^\vee) = \Im g^\vee$ we conclude that $\|g^\vee\|_{L^{2k}(\mu)} = O(\sqrt{k}\|g\|_{\ell^2(\Lambda)})$ and hence that the map $g \mapsto g^\vee$ from $\ell^2(\Lambda)$ to $L^{2k}(\mu)$ has norm $O(\sqrt{k})$.  

Now, as with the Bessel bound, by duality and convexity of the $L^p(\mu)$-norms we see that
\begin{equation*}
|\Lambda| \epsilon^2\|f\|_{L^1(\mu)}^2 \leq \sum_{\lambda \in \Lambda}{|(fd\mu)^\wedge(\lambda)|^2} = O(k \|f\|_{L^{2k/(2k-1)}(\mu)}^2) =O(k \|f\|_{L^1(\mu)}^{2}L_f^{2/k}).
\end{equation*}
The result follows on putting $k=\lceil \log 2L_f\rceil $ and rearranging.
\end{proof}
Although Chang's theorem cannot be significantly improved, there are some small refinements and discussions of their limitations in the work \cite{shk::,shk::2} and \cite{shk::4} of Shkredov.  Our approach is largely insensitive to these developments and so we shall not concern ourselves with them here.

Instead of trying to improve Chang's theorem we are going to show that if the large spectrum of a function contains many dissociated elements then we get correlation with a Riesz product.  The result should be compared with Proposition (*) from \cite{bou::1} which has a very readable proof in that paper.

On a first reading one may wish to take $\mu=\mu'$.  When we use the lemma we shall take $\mu'$ positive definite and such that $\mu \ast \mu' \approx \mu$.  It then follows that the integral on the right of the second conclusion is roughly an upper bound for the constant of dissociativity of $\Lambda'''$ w.r.t. $\mu$.  (The reader may wish to look forwards to Lemma \ref{lem.ctsriesz} to get some idea of what we need the lemma below to dovetail with.)
\begin{lemma} \label{lem.newinc} Suppose $\mu'$ is another probability measure, $A$ is an event with $\alpha:=\mu(A)>0$, $\Lambda\subset \Spec_\tau(1_A,\mu)$ is $\eta$-orthogonal w.r.t. $\mu'$ and has size $k$, and $m\geq 1$ is a parameter.  Then at least one of the following is true:
\begin{enumerate}
\item there is a set $\Lambda'$ of $(1,\mu')$-relative entropy $O(m\log 2k)$ which contains at least $k/2$ of the elements of $\Lambda$;
\item for any $d$ with $C_{\ref{lem.newinc}}\log 2\tau^{-1} \leq d \leq  c_{\ref{lem.newinc}}\min\{\tau k,\sqrt{m}\}$ there is some set $\Lambda'''$ of $O(d)$ characters and function $\omega':\Lambda'''\rightarrow D$ such that
\begin{equation*}
\langle 1_A,p_{\omega',\Lambda'''}\rangle_{L^2(\mu)} \geq \alpha(1+\tau d/4)\int{p_{1,\Lambda'''}d\mu'};
\end{equation*}
\end{enumerate}
provided $\eta \leq c_{\ref{lem.newinc}}'\tau$.
\end{lemma}
\begin{proof}
Let $l:=\lceil 2m\log 2k\rceil$ and decompose $\Lambda$ as a disjoint union of sets $\Lambda',\Lambda_1,\dots,\Lambda_r$ where 
\begin{enumerate}
\item any subset of $\Lambda'$ that is $1$-dissociated w.r.t. $\mu'$ has size at most $l$;
\item every $\Lambda_i$ has size $l$ and is $1$-dissociated w.r.t. $\mu'$.
\end{enumerate}
If $|\Lambda'|$ is at least $k/2$ then we are in the first case of the lemma; thus we may assume that $\Lambda'':=\Lambda_1\sqcup \dots \sqcup \Lambda_r$ has size at least $k/2$.

Let $\omega: \Lambda''\rightarrow D$ be such that
\begin{equation*}
\omega(\lambda)(1_Ad\mu)^\wedge(\lambda) = |(1_Ad\mu)^\wedge(\lambda) | \textrm{ for all } \lambda \in \Lambda'';
\end{equation*}
and let $(X_\lambda)_{\lambda \in \Lambda''}$ be independent identically distributed random variables with
\begin{equation*}
\P(X_\lambda=1) = \theta \textrm{ and } \P(X_\lambda=0)=1-\theta, \textrm{ where } \theta:=d/k,
\end{equation*}
possible by assuming $d \leq k$.  Since the $X_\lambda$s are independent we have that
\begin{equation}\label{eqn.ind}
\E p_{\omega''X,\Lambda''}(x)=  \prod_{\lambda \in \Lambda''}{\E (1+X_\lambda \Re (\omega''(\lambda)\lambda(x)))}  = \prod_{\lambda \in \Lambda''}{(1+\theta \Re( \omega''(\lambda)\lambda(x)))},
\end{equation}
for any $\omega'':\Lambda''\rightarrow D$.  We write $\mathcal{N}$ for the event that at most $2e d$ of the variables are $1$ and specialise to $\omega''=\omega$ so that again, by independence,
\begin{eqnarray*}
\E 1_\mathcal{N} p_{\omega X,\Lambda''}(x) & = & \E p_{\omega X,\Lambda''}(x)-\sum_{r>2ed}{\theta^r\sum_{S \subset \Lambda'': |S|=r}{ \prod_{\lambda \in S}{\Re( \omega(\lambda)\lambda(x))}}}\\ & \geq & \E p_{\omega X,\Lambda''}(x) - \sum_{r > 2ed}{\theta^r\binom{|\Lambda''|}{r}}\\ & \geq & \E p_{\omega X,\Lambda''}(x) - \exp(-\Omega(d))
\end{eqnarray*}
since $\binom{n}{m} \leq (en/m)^m$ for all non-negative integers $n$ and $m$.

Of course, $1+z \geq \exp(z-2z^2)$ whenever $-1/2 \leq z \leq 1/2$, and $\exp(z) \geq 1+z$ for all $z$  so we have (since we may assume $\theta \leq 1/2$) that
\begin{eqnarray*}
\E p_{\omega X,\Lambda''}(x)&\geq &\prod_{\lambda \in \Lambda''}{\exp(\theta \Re( \omega(\lambda)\lambda(x)) - 2\theta^2)}\\  &=& \exp(-2\theta^2k)\exp(\theta \Re \omega^\vee(x))\geq  \exp(-2\theta^2k)(1+\theta \Re \omega^\vee(x)).
\end{eqnarray*}
Coupled with non-negativity of $1_A$ this means that
\begin{eqnarray}
\nonumber \E1_\mathcal{N}\langle 1_A,p_{\omega X,\Lambda''}\rangle_{L^2(\mu)} & =& \langle 1_A,\E p_{\omega X,\Lambda''}\rangle_{L^2(\mu)}-\alpha \exp(-\Omega(d))\\ \nonumber & \geq & \exp(-2\theta^2k)(\alpha + \theta \sum_{\lambda \in \Lambda''}{|(1_Ad\mu)^\wedge(\lambda)|})-\alpha \exp(-\Omega(d)) \\ \label{eqn.key}& \geq & (1-2\theta^2k)(\alpha + \theta \tau \alpha k/2)-\alpha \exp(-\Omega(d)).
\end{eqnarray}

In the other direction we specialise to $\omega''\equiv 1$ in (\ref{eqn.ind}) and get
\begin{equation*}
\E p_{X,\Lambda''}(x)=  \prod_{\lambda \in \Lambda''}{(1+\theta \Re \lambda(x))} \leq  \exp(\theta \sum_{\lambda \in \Lambda''}{\Re \lambda(x)}) = \exp(\theta \Re 1_{\Lambda''}^\vee(x)).
\end{equation*}
However, $\exp(z) \leq 1+z+z^2\exp(z)$ for all $z\in \R$ so that
\begin{equation}\label{eqn.main}
\E \int{p_{X,\Lambda''}d\mu'} \leq  1+\int{\theta \Re 1_{\Lambda''}^\vee d\mu'} + \int{(\theta \Re 1_{\Lambda''}^\vee)^2\exp(\theta \Re 1_{\Lambda''}^\vee)d\mu'}.
\end{equation}
The second term on the right is estimated using the $\eta$-orthogonality of $ \Lambda$.  Indeed, suppose that $|z|=1$, then
\begin{eqnarray*}
2z\int{\Re1_{\Lambda''}^\vee d\mu'} & = & \sum_{\lambda \in \Lambda''}{2\Re\int{z\lambda d\mu'}}\\ & = & \sum_{\lambda \in \Lambda''}{\left(\int{|1+z\lambda|^2d\mu'} - 2\right)} \leq 2\eta k
\end{eqnarray*}
since $|\Lambda''| \leq k$.  Thus, on suitable choice of $z$, we see that the relevant term in (\ref{eqn.main}) is at most $\eta k$ in absolute value.  To estimate the third term on the right in (\ref{eqn.main}) we apply H{\"o}lder's inequality with index $q:=1+\log k$ and conjugate index $q'$ to get that
\begin{equation*}
 \int{(\theta \Re 1_{\Lambda''}^\vee)^2\exp(\theta \Re 1_{\Lambda''}^\vee)d\mu'} \leq \theta^2\left(\int{|\Re 1_{\Lambda''}^\vee|^{2q'}d\mu'}\right)^{1/q'}\left(\int{\exp(\theta q \Re 1_{\Lambda''}^\vee)d\mu'}\right)^{1/q}.
\end{equation*}
Now $|\Re 1_{\Lambda''}^\vee| \leq k$ and
\begin{equation*}
\int{|\Re 1_{\Lambda''}^\vee|^2d\mu'} \leq \int{|1_{\Lambda''}^\vee|^2d\mu'} \leq 2k
\end{equation*}
since $\Lambda$ is certainly $1$-orthogonal.  It follows that
\begin{equation*}
\left(\int{|\Re 1_{\Lambda''}^\vee|^{2q'}d\mu'}\right)^{1/q'} \leq ((2k).k^{2q'-2})^{1/q'} = O(k),
\end{equation*}
and combining what we have so far gives
\begin{equation*}
\E \int{p_{X,\Lambda''}d\mu'} \leq  1+O(\eta \theta k) + O\left(\theta^2k \left(\int{\exp(\theta q \Re 1_{\Lambda''}^\vee)d\mu'}\right)^{1/q}\right).
\end{equation*}
Recalling that each $\Lambda_i$ is $1$-dissociated w.r.t. $\mu'$ and that $\Lambda''$ is their disjoint union it follows from Lemma \ref{lem.chernoff} that
\begin{eqnarray*}
\int{\exp(\theta q \Re 1_{\Lambda''}^\vee)d\mu'}& \leq & \prod_{i=1}^r{ \left(\int{\exp(rq\theta \Re1_{\Lambda_i}^\vee)d\mu'}\right)^{1/r}}\\ & \leq & \prod_{i=1}^r{\exp(1/r)\exp(rq^2\theta^2\|1_{\Lambda_i}\|_{\ell^2(\Lambda_i)}^2/2)}.
\end{eqnarray*}
Now, $\|1_{\Lambda_i}\|_{\ell^2(\Lambda_i)}^2= l$ and $rl \leq k$, so
\begin{equation*}
\E \int{p_{X,\Lambda''}d\mu'} \leq 1 + O(\theta \eta k) + O(\theta^2k\exp( k^2q\theta^2/2l)).
\end{equation*}
Inserting this into (\ref{eqn.key}) we get that
\begin{eqnarray*}
\E 1_{\mathcal{N}} \langle 1_A,p_{\omega X,\Lambda''}\rangle_{L^2(\mu)}& \geq &\alpha \left((1-O(\eta d) -O(d^2k^{-1}\exp(d^2q/2l)))(1+\tau d/2)\right.\\ & & \left. - \exp(-\Omega(d))\right)\E \int{p_{X,\Lambda''}d\mu'} .
\end{eqnarray*}
It follows that there are absolute constants $c,c'$ and $C$ such that if
\begin{equation*}
\eta \leq c\tau,\mbox{ } dk^{-1}\exp(d^2q/2l) \leq c'\tau \textrm{ and } d \geq C\log 2\tau^{-1}
\end{equation*}
then
\begin{equation*}
\E 1_{\mathcal{N}} \langle 1_A,p_{\omega X,\Lambda''}\rangle_{L^2(\mu)}> \E\alpha(1+\tau d/4)\int{p_{X,\Lambda''}d\mu'}.
\end{equation*}
By averaging there is some element such that the integrand on the left is at least that on the right and the result follows on setting $\Lambda'''=\{\lambda \in \Lambda'':X_\lambda=1\}$, and $\omega'=\omega|_{\Lambda'''}$.
\end{proof}
It is possible to optimise this lemma slightly more efficiently but we shall not concern ourselves with that here.

\section{Basic properties of Bohr sets}\label{sec.bohr}

In this section we collect the basic facts which we need about Bohr sets.  Although this material has all now become standard this is in no small part due to the expository work of Tao; \cite[Chapter 4.4]{taovu::} is the recommended reference.

To start with we have the following averaging argument \emph{c.f.} \cite[Lemma 4.20]{taovu::}.
\begin{lemma}[Size of Bohr sets]\label{lem.bohrsize} Suppose that $B$ is a Bohr set.  Then
\begin{equation*}
\mu_G(B_{2\rho}) =\exp(O(\rk(B)))\mu_G(B_\rho) \textrm{ for all } \rho \in \R^+
\end{equation*}
and
\begin{equation*}
\mu_G(B) \geq \exp(-O(\rk(B)))\prod_{\gamma \in \Gamma}{\delta_\gamma}.
\end{equation*}
\end{lemma}
Given two Bohr sets $B$ and $B'$ we define their \emph{intersection} to be the Bohr set with frequency set $\Gamma \cup \Gamma'$ and width function $\delta \wedge \delta'$.  It follows immediately from the previous lemma that Bohr sets are well behaved under intersections.
\begin{lemma}[Intersections of Bohr sets]
Suppose that $B$ and $B'$ are Bohr sets.  Then
\begin{equation*}
\rk (B\wedge B') \leq \rk(B) + \rk(B')
\end{equation*}
and
\begin{equation*}
\mu_G(B\wedge B') \geq \exp(-O(\rk(B)+\rk(B')))\mu_G(B)\mu_G(B').
\end{equation*}
\end{lemma}
\begin{proof}
The rank inequality is obvious.  For the density note that 
\begin{equation*}
(B \wedge B')_1 \supset (B_{1/2}-B_{1/2}) \cap (B'_{1/2}-B'_{1/2})
\end{equation*}
by the triangle inequality.  On the other hand
\begin{equation*}
\mu_G((B_{1/2}-B_{1/2}) \cap (B'_{1/2}-B'_{1/2}))\mu_G(B_{1/2})\mu_G(B'_{1/2})
\end{equation*}
is at least
\begin{equation*}
\int{1_{B_{1/2}} \ast 1_{B_{1/2}}1_{B'_{1/2}}\ast 1_{B'_{1/2}}d\mu_G}= \|1_{B_{1/2}} \ast 1_{B_{1/2}'}\|_{L^2(\mu_G)}^2\geq (\mu_G(B_{1/2})\mu_G(B'_{1/2}))^2,
\end{equation*}
where the last inequality is Cauchy-Schwarz.  The result now follows from the first case of the previous lemma.
\end{proof}

It is the first condition of Lemma \ref{lem.bohrsize} which is the most important and informs the definition of the dimension of a Bohr set: a Bohr set $B$ is said to be \emph{$d$-dimensional} if
\begin{equation*}
\mu_G(B_{2\rho}) \leq 2^d\mu_G(B_\rho) \textrm{ for all } \rho \in (0,1],
\end{equation*}
and it will be convenient to assume that it is always at least $1$.  Since any $d$-dimensional Bohr set is always $d'$-dimensional for all $d'\geq d$ this is not a problem.

Note, in particular, that by Lemma \ref{lem.bohrsize} a Bohr set of rank $k$ has dimension $O(k)$.  It is, however, the dimension which is the important property of Bohr sets and the only reason we mention the rank is that it is sub-additive with respect to intersection, unlike dimension.

Some Bohr sets behave better than others: a $d$-dimensional Bohr set is said to be \emph{$C$-regular} if
\begin{equation*}
\frac{1}{1+Cd |\eta|} \leq \frac{\mu_G(B_{1+\eta})}{\mu_G(B_1)} \leq 1+Cd|\eta| \textrm{ for all } \eta \textrm{ with } |\eta| \leq 1/Cd.
\end{equation*}
Crucially, regular Bohr sets are plentiful:
\begin{lemma}[Regular Bohr sets]\label{lem.ubreg} There is an absolute constant $C_\mathcal{R}$ such that whenever $B$ is a Bohr set, there is some $\lambda \in [1/2,1)$ such that $B_\lambda$ is $C_\mathcal{R}$-regular.
\end{lemma}
The result is proved by a covering argument due to Bourgain \cite{bou::5}; for details one may also consult \cite[Lemma 4.25]{taovu::}.  For the remainder of the paper we shall say \emph{regular} for $C_\mathcal{R}$-regular.

\section{Bohr sets as majorants}

One of the key properties of Bohr sets is that they can be used to majorise sets of characters.  They are particularly good at this because if two sets are majorised by two different Bohr sets then their sumset is majorised by (a dilation of) the intersection of the Bohr sets.

We begin by recalling a lemma which shows that a set majorised by a Bohr set is approximately annihilated by a dilate, and if a set is approximately annihilated by a Bohr set then it is certainly majorised by it.
\begin{lemma}[Majorising and annihilating]\label{lem.nest}
Suppose that $B$ is a regular $d$-dimensional Bohr set.  Then
\begin{equation*}
\{\gamma:|\wh{\beta}(\gamma)| \geq \kappa\} \subset \{\gamma: |1-\gamma(x)| \leq C_{\ref{lem.nest}}\rho \kappa^{-1}d \textrm{ for all } x \in B_\rho\},
\end{equation*}
 and
\begin{equation*}
\{\gamma:|1-\gamma(x)| \leq \eta \textrm{ for all } x \in B\} \subset \{\gamma: |\wh{\beta}(\gamma)| \geq 1-\eta\}.
\end{equation*}
\end{lemma}
\begin{proof}
First, suppose that $|\wh{\beta}(\gamma)| \geq \kappa$ and $y \in B_\rho$.  Then
\begin{equation*}
|1-\gamma(y)|\kappa \leq |\int{\gamma(x)d\beta} - \int{\gamma(x+y)d\beta(x)}|\leq \frac{\mu_G(B_{1+\rho}\setminus B_{1-\rho})}{\mu_G(B)} = O(d\rho)
\end{equation*}
provided $\rho \leq 1/C_\mathcal{R}d$.  The first inclusion follows; the second is a trivial application of the triangle inequality.
\end{proof}
The proof is from \cite[Lemma 3.6]{grekon::} where the significance of this result seems to have first been clearly recognised.

Of particular interest to us is the majorising of orthogonal and dissociated sets and in light of the previous lemma that can be achieved by the following.  Recall that if $X$ is a set then $\mathcal{P}(X)$ denotes the \emph{power set}, that is the set of all subsets of $X$.
\begin{lemma}[Annihilating orthogonal sets]\label{lem.majortho}
Suppose that $B$ is a regular $d$-dimensional Bohr set and $\Delta$ is a set of characters with $(\eta,\beta)$-relative size $k$.  Then there is a set $\Lambda$, $\eta$-orthogonal w.r.t. $\beta$, and an order-preserving map $\phi:\mathcal{P}(\Lambda) \rightarrow \mathcal{P}(\Delta)$ such that $\phi(\Lambda)=\Delta$ and for any $\Lambda' \subset \Lambda$ and $\gamma \in \phi(\Lambda')$ we have
\begin{equation*}
|1-\gamma(x)| \leq C_{\ref{lem.majortho}}(\nu +\rho d^2\eta^{-1} \log 2k\eta^{-1}) \textrm{ for all } x \in B_\rho \wedge B'_\nu
\end{equation*}
where $B'$ is the Bohr set with constant width function $2$ and frequency set $\Lambda'$.
\end{lemma}
\begin{proof}
Let $L:=\lceil \log_2 2(k+1)^{3/2}\eta^{-1}\rceil$, the reason for which choice will become apparent, and define
\begin{equation*}
\beta^+:=\beta_{1+L{\rho'}} \ast \beta_{\rho'} \ast \dots \ast \beta_{\rho'},
\end{equation*}
where $\beta_{\rho'}$ occurs $L$ times in the expression.  By regularity (of $B$) we can pick ${\rho'}\in (\Omega(\eta/dL),1]$ such that $B_{\rho'}$ is regular and $\mu_G(B_{1+L{\rho'}}) \leq (1+\eta/3)\mu_G(B)$.  On the other hand we have the point-wise inequality
\begin{equation*}
\mu_G(B)\beta = 1_B \leq 1_{B_{1+{L\rho'}}} \ast \beta_{\rho'} \ast \dots \ast \beta_{\rho'} = \mu_G(B_{1+L\rho'})\beta^+,
\end{equation*}
whence $\beta \leq 1+\eta/3)\beta^+$.  It follows that if $\Lambda$ is $\eta/2$-orthogonal w.r.t. $\beta^+$ then $\Lambda$ is $\eta$-orthogonal w.r.t. $\beta$, and hence $\Lambda$ has size at most $k$.  From now on all orthogonality will be w.r.t. $\beta^+$. 

We put $\eta_i:=i\eta/2(k+1)$ and begin by defining a sequence of sets $\Lambda_0,\Lambda_1,\dots$ iteratively such that $\Lambda_i$ is $\eta_i$-orthogonal.  We let $\Lambda_0:=\emptyset$ which is easily seen to be $0$-orthogonal.  Now, suppose that we have defined $\Lambda_i$ as required.  If there is some $\gamma \in \Delta\setminus \Lambda_i$ such that $\Lambda_i \cup \{\gamma\}$ is $\eta_{i+1}$-orthogonal then let $\Lambda_{i+1}:=\Lambda_i \cup \{\gamma\}$.  Otherwise, terminate the iteration.

Note that for all $i \leq k+1$, if the set $\Lambda_i$ is defined then it is certainly $\eta/2$-orthogonal and so $|\Lambda_i| \leq k$.  However, if the iteration had continued for $k+1$ steps then $|\Lambda_{k+1}| >k$.  This contradiction means that there is some $i\leq k$ such that $\Lambda:=\Lambda_i$ is $\eta_i$-orthogonal and $\Lambda_i \cup \{\gamma\}$ is not $\eta_{i+1}$-orthogonal for any $\gamma \in \Delta \setminus \Lambda_i$.

Now, we define $\phi$ in $\mathcal{P}(\Lambda)$ by
\begin{equation*}
\phi(\Lambda') = \{\gamma \in \Delta: |\wh{\beta_{\rho'}}(\gamma-\lambda)| \geq 1/2 \textrm{ for some } \lambda \in \Lambda' \cup\{0_{\wh{G}}\}\}.
\end{equation*}
It is immediate that $\phi$ is order-preserving.  Moreover we have the following claim.
\begin{claim*}
$ \phi(\Lambda) = \Delta$.
\end{claim*}
\begin{proof}
If $\lambda \in \Lambda$ then $\lambda \in \phi(\Lambda)$ since $\wh{\beta_{\rho'}}(0_{\wh{G}})=1$, so we need only consider $\gamma \in \Delta \setminus \Lambda$.  In that case there is some $g \in \ell^2(\Lambda)$ and $\nu \in \C$ such that
\begin{equation*}
\int{|1+g^\vee + \nu \gamma |^2d\beta^+} > (1+\eta_{i+1})(1+\|g\|_{\ell^2(\Lambda)}^2 + |\nu|^2).
\end{equation*}
It follows from the $\eta_i$-orthogonality of $\Lambda$ that
\begin{eqnarray*}
2\Re \langle 1+g^\vee,\nu\gamma \rangle_{L^2(\beta^+)} & > &(1+\eta_{i+1})(1+\|g\|_{\ell^2(\Lambda)}^2 + |\nu|^2)\\ & & -(1+\eta_i)(1+\|g\|_{\ell^2(\Lambda)}^2) - |\nu|^2\\ & \geq & \frac{\eta}{2(k+1)}(1+\|g\|_{\ell^2(\Lambda)}^2 + |\nu|^2)\\ & \geq & \frac{\eta}{(k+1)}(1+\|g\|_{\ell^2(\Lambda)}^2)^{1/2}|\nu|.
\end{eqnarray*}
Thus
\begin{equation*}
|\langle 1+ g^\vee,\gamma \rangle_{L^2(\beta^+)}| \geq  \frac{\eta}{2(k+1)}(1+\|g\|_{\ell^2(\Lambda)}^2)^{1/2}
\end{equation*}
By Plancherel's theorem this tells us that
\begin{eqnarray*}
 \frac{\eta}{2(k+1)} (1+\|g\|_{\ell^2(\Lambda)}^2)^{1/2}& \leq & \left|\wh{\beta^+}(\gamma)+\sum_{\lambda \in \Lambda}{g(\lambda)\wh{\beta^+}(\gamma-\lambda)}\right|\\ & \leq & (1+\|g\|_{\ell^1(\Lambda)})\sup_{\lambda \in \Lambda\cup \{0_{\wh{G}}\}}{|\wh{\beta_{\rho'}}(\gamma-\lambda)|^L}.
\end{eqnarray*}
By Cauchy-Schwarz we have $1+\|g\|_{\ell^1(\Lambda)} \leq \sqrt{k+1}(1+\|g\|_{\ell^2(\Lambda)}^2)^{1/2}$, and so it follows (given the size of $L$) that there is some $\lambda \in \Lambda\cup\{0_{\wh{G}}\}$ such that $|\wh{\beta_{{\rho'}}}(\gamma-\lambda)| \geq 1/2$.
\end{proof}
Now, suppose that $\Lambda' \subset \Lambda$ and $\gamma \in \phi(\Lambda')$.  It follows that there is some $\lambda \in \Lambda' \cup\{0_{\wh{G}}\}$ such that $|\wh{\beta_{\rho'}}(\gamma-\lambda)|\geq 1/2$ and thus by Lemma \ref{lem.nest} we have that
\begin{equation*}
|1-\gamma(x)\overline{\lambda}(x)| \leq 2C_{\ref{lem.nest}}\rho'^{-1}\rho d \textrm{ for all } x \in B_\rho.
\end{equation*}
On the other hand by definition $|1-\lambda(x)| =O(\nu)$ if $x \in B'_\nu$ and so the result follows from the triangle inequality.
\end{proof}
The argument for dissociated sets is rather similar so we omit many of the details.
\begin{lemma}[Annihilating dissociated sets]\label{lem.majdissoc}
Suppose that $B$ is a regular $d$-dimensional Bohr set and $\Delta$ is a set of characters with $(\eta,\beta)$-relative entropy $k$.  Then there is a set $\Lambda$, $\eta$-dissociated w.r.t. $\beta$, such that for all $\gamma \in \Delta$ we have
\begin{equation*}
|1-\gamma(x)| \leq C_{\ref{lem.majdissoc}}(k \nu + \rho d^2 \eta^{-1}(k +\log 2\eta^{-1})) \textrm{ for all }x \in B_\rho \wedge B'_\nu
\end{equation*}
where $B'$ is the Bohr set with constant width function $2$ and frequency set $\Lambda$.
\end{lemma}
\begin{proof}
By the same argument as at the start of the previous lemma, but with $L=\lceil \log_2 3^k2(k+1)\eta^{-1}\rceil$, we get some $i \leq k$, an $\eta_i$-dissociated (w.r.t. $\beta^+$) set $\Lambda$ of at most $k$ characters, such that for all $\gamma \in \Delta \setminus \Lambda$ there is a function $\omega:\Lambda \rightarrow D$ and $\nu \in D$ such that
\begin{equation*}
\int{p_{\omega,\Lambda}(1+\Re(\nu \gamma))d\beta^+} > \exp(\eta_{i+1}).
\end{equation*}

Now, suppose that $\gamma \in \Delta$.  If $\gamma \in \Lambda$ then the conclusion is immediate, so we may assume that $\gamma \in \Delta \setminus \Lambda$.  Then, since $\Lambda$ is $\eta_i$-dissociated, we see that
\begin{equation*}
|\int{p_{\omega,\Lambda}\overline{\gamma} d\beta^+}|  >  \exp(\eta_{i+1})- \exp(\eta_i) \geq \frac{\eta}{2(k+1)}.
\end{equation*}
Applying Plancherel's theorem we get that
\begin{equation*}
 \frac{\eta}{2(k+1)} \leq \left|\sum_{\lambda \in \Span(\Lambda)}{\wh{p_{\omega,\Lambda}}(\lambda)\wh{\beta^+}(\gamma-\lambda)}\right|\leq 3^k\sup_{\lambda \in \Span(\Lambda)}{|\wh{\beta_\rho}(\gamma-\lambda)|^L}.
\end{equation*}
Given the choice of $L$ there is some $\lambda \in \Span(\Lambda)$ such that $|\wh{\beta_{\rho'}}(\gamma-\lambda)| \geq 1/2$.  The argument now proceeds as before on noting that $|1-\lambda(x)|=O(k\nu)$ if $x \in B'_\nu$ by the triangle inequality.
\end{proof}
It is also useful to record an immediate corollary of this.
\begin{corollary}\label{cor.dis}
Suppose that $B$ is a regular $d$-dimensional Bohr set and $\Delta$ is a set of characters with $(\eta,\beta)$-relative entropy $k$.  Then there is a Bohr set $B' \leq B$ with
\begin{equation*}
\rk(B') \leq \rk(B)+k \textrm{ and } \mu_G(B') \geq (\eta/2dk)^{O(d)}(1/2k)^{O(k)}\mu_G(B)
\end{equation*}
such that $|1-\gamma(x)| \leq 1/2$ for all $x \in B'$ and $\gamma \in \Delta$.
\end{corollary}

\section{Getting a density increment}

The basic dichotomy of the density increment strategy is driven by the following lemma which says that either we are in a situation where we have many three-term progressions or else we have a concentration of Fourier mass.
\begin{lemma}\label{lem.fdic}
Suppose that $B$ is a regular $d$-dimensional Bohr set and $A \subset B$ and $A' \subset B_\rho$ have density $\alpha>0$ and $\alpha'>0$ respectively.  Then either
\begin{enumerate}
\item (Many three-term progressions)
\begin{equation*}
T(1_A,1_{A'},1_A) \geq \alpha^2\alpha'\mu_G(B)\mu_G(B_\rho)/2;
\end{equation*}
\item (Concentration of Fourier mass)
\begin{equation*}
\sum_{-2\gamma \in \Spec_{c_{\ref{lem.fdic}}\alpha}(1_{A'},\beta_\rho)}{|((1_A-\alpha)1_B)^\wedge(\gamma)|^2|(1_{A'}d\beta_\rho)^\wedge(-2\gamma)|}\geq c_{\ref{lem.fdic}}\alpha^2\alpha'\mu_G(B);
\end{equation*}
\end{enumerate}
provided $\rho \leq c_{\ref{lem.fdic}}'\alpha/d$.
\end{lemma}
\begin{proof}
Write $M:L^2(\beta) \rightarrow L^2(\beta)$ for the linear operator taking $f \in L^2(\beta)$ to
\begin{equation*}
x \mapsto \int{f(x-2y)1_{A'}(y)d\beta_\rho(y)}1_B(x),
\end{equation*}
and put
\begin{equation*}
T(f):=\langle Mf,f\rangle_{L^2(\beta)},
\end{equation*}
so that
\begin{equation*}
T(1_A,1_{A'},1_A)=\mu_G(B)\mu_G(B_\rho)T(1_A).
\end{equation*}
We begin by noting that
\begin{eqnarray*}
T(1_A) & =&  T((1_A - \alpha)1_B) + \alpha \langle M1_A,1_B\rangle_{L^2(\beta)} \\ & & + \alpha \langle M1_B,1_A\rangle_{L^2(\beta)} - \alpha^2\langle M1_B,1_B\rangle_{L^2(\beta)}.
\end{eqnarray*}
By regularity we have
\begin{equation*}
\|M1_B - \alpha' 1_B\| =O(\rho d \alpha'\mu_G(B))
\end{equation*}
provided $\rho \leq 1/2C_{\mathcal{R}}d$.  Thus, by the triangle inequality we have
\begin{equation*}
T(1_A) = T((1_A - \alpha)1_B) + \alpha^2\alpha' + O(\alpha \alpha' \rho d).
\end{equation*}
It follows that there is some absolute $c>0$ such that if $\rho \leq c \alpha /d$ then either we are in the first case of the lemma or else
\begin{equation*}
|T((1_A - \alpha)1_B)| \geq \alpha^2\alpha'/4;
\end{equation*}
we may naturally assume the latter or else we are done.  The usual application of the transform tells us that
\begin{equation*}
\sum_{\gamma \in\wh{G}}{|((1_A-\alpha)1_B)^\wedge(\gamma)|^2|(1_{A'}d\beta_\rho)^\wedge(-2\gamma)|} \geq \alpha^2\alpha'\mu_G(B)/4,
\end{equation*}
from which it follows that we are in the second case of the lemma by the triangle inequality.
\end{proof}

The remaining two lemmas of the section encode the other half of the density increment strategy: having analysed the Fourier mass in some way, we need results which convert this analysis into a density increment.

First we have the Heath-Brown-Szemer{\'e}di energy increment technique.  There are more general versions of the lemma suitable for other problems; the one here is adapted to our circumstance and, in particular, is the only place we use the fact that $G$ has odd order.
\begin{lemma}[Energy to density lemma]\label{lem.basicl2}
Suppose that $B$ is a regular $d$-dimensional Bohr set, $A \subset B$ has density $\alpha>0$, $\mathcal{L}$ is a set of characters such that
\begin{equation*}
\sum_{-2\gamma \in \mathcal{L}}{|((1_A-\alpha)1_B)^\wedge(\gamma)|^2} \geq K \alpha^2\mu_G(B),
\end{equation*}
and $B'\leq B_\rho$ is a $d'$-dimensional Bohr set such that
\begin{equation*}
\mathcal{L} \subset \{\gamma : |1-\gamma(x)| \leq 1/2 \textrm{ for all } x \in B'\}.
\end{equation*}
Then there is a regular Bohr set $B''$ with
\begin{equation*}
\rk(B'')=\rk(B') \textrm{ and } \mu_G(B'') \geq 2^{-O(d')}\mu_G(B')
\end{equation*}
such that
\begin{equation*}
\|1_A \ast \beta''\|_{L^\infty(\mu_G)} \geq \alpha (1+c_{\ref{lem.basicl2}}K)
\end{equation*}
provided $\rho \leq c_{\ref{lem.basicl2}}'\alpha K/d$.
\end{lemma}
\begin{proof}
Put $B''':=-2.B_{1/2}'$, and note that $B'''$ has the same rank and size as $B_{1/2}'$ (since $G$ has odd order), and 
\begin{equation*}
\{ \gamma : -2\gamma \in \mathcal{L}\} \subset \{\gamma: |1-\gamma(-2x)| \leq 1/2 \textrm{ for all }x \in B'\}.
\end{equation*}
Pick $\nu \in [1/2,1)$ such that $B'':=B_\nu'''$ is regular and so, since $\{-2x: x \in B'\} \supset \{x: x \in B''\}$, we conclude that
\begin{equation*}
\{ \gamma : -2\gamma \in \mathcal{L}\} \subset \{\gamma: |1-\gamma(x)| \leq 1/2 \textrm{ for all } x \in B''\}.
\end{equation*}
By the triangle inequality if $\gamma$ is in this second set, then $|\wh{\beta''}(\gamma)| \geq 1/2$.  Thus Plancherel's theorem tells us that
\begin{eqnarray*}
\|((1_A-\alpha)1_B) \ast \beta''\|_{L^2(\mu_G)}^2& = &\sum_{\gamma \in \wh{G}}{|((1_A-\alpha)1_B)^\wedge(\gamma)|^2|\wh{\beta''}(\gamma)|^2}\\ & \geq & \frac{1}{4}\sum_{-2\gamma \in \mathcal{L}}{|((1_A-\alpha)1_B)^\wedge(\gamma)|^2}\geq \frac{K}{4} \alpha^2\mu_G(B).
\end{eqnarray*}
But $B'' \subset B_1' \subset B_\rho$, so by regularity we have that
\begin{equation*}
\langle f,1_B \ast \beta'' \ast \beta'' \rangle_{L^2(\mu_G)} = (\int{fd\beta} + O(\rho d\|f\|_{L^\infty(\beta)}))\mu_G(B)
\end{equation*}
for any $f \in L^\infty(\beta)$ provided $\rho \leq 1/2C_\mathcal{R}d$.  Thus
\begin{eqnarray*}
\|((1_A - \alpha)1_B) \ast \beta''\|_{L^2(\mu_G)}^2& =& \|1_A\ast \beta''\|_{L^2(\mu_G)}^2 - 2\alpha((\alpha + O(\rho d))\mu_G(B)\\ & & +\alpha^2(1+O(\rho d))\mu_G(B),
\end{eqnarray*}
whence
\begin{equation*}
 \|1_A\ast \beta''\|_{L^2(\mu_G)}^2 \geq \alpha^2\mu_G(B) + K\alpha^2\mu_G(B)/4 - O(\alpha\rho d \mu_G(B)).
\end{equation*}
The result now follows provided $\rho$ is sufficiently small on applying H{\"o}lder's inequality and dividing by $\alpha \mu_G(B)$.
\end{proof}
The second lemma we need is used for leveraging of the second conclusion of Lemma \ref{lem.newinc} and gives a way of passing between correlation with a Riesz product and a density increment on a Bohr set.
\begin{lemma}[Riesz product correlation to density lemma]\label{lem.ctsriesz}
Suppose that $B$ is a regular $d$-dimensional Bohr set, $A \subset B_{1-3\rho}$, $\Lambda$ is a set of $k$ characters and $\omega:\Lambda \rightarrow D$ is such that
\begin{equation*}
\langle 1_A,p_{\omega,\Lambda}\rangle_{L^2(\beta)} \geq \alpha(1+\epsilon)\int{p_{1,\Lambda}d\mu},
\end{equation*}
where $\mu=\beta_\rho\ast \beta_\rho$.  Then there is a regular Bohr set $B'$ with
\begin{equation*}
\rk(B') \leq \rk(B) + k \textrm{ and } \mu_G(B') \geq (\epsilon\alpha/2k)^{O(k)}2^{-O(d)}\mu_G(B_\rho)
\end{equation*}
such that $\|1_A \ast \beta'\|_{L^\infty(\mu_G)} \geq \alpha(1+\epsilon/2)$.
\end{lemma}
\begin{proof}
Let $B''$ be the Bohr set with frequency set $\Lambda$ and width function equal to the constant function $1$.  Let $\lambda_1,\dots,\lambda_k$ be an enumeration of $\Lambda$ and put
\begin{equation*}
\omega_i:\Lambda \rightarrow D; \lambda \mapsto \begin{cases}\omega_\lambda & \textrm{ if } \lambda=\lambda_j \textrm{ for some } j>i\\ 0 & \textrm{ if } \lambda=\lambda_i\\  \omega_\lambda \lambda(y) & \textrm{ if } \lambda=\lambda_j \textrm{ for some }j<i.\end{cases}
\end{equation*}
With these definitions we have
\begin{equation*}
p_{\omega,\Lambda}(x+y)-p_{\omega,\Lambda}(x) = \sum_{i=1}^k{(\Re( \omega_{\lambda_i}\lambda_i(x+y))-\Re (\omega_{\lambda_i}\lambda_i(x)) )p_{\omega_i,\Lambda}(x)},
\end{equation*}
but
\begin{equation*}
|\Re (\omega_{\lambda_i}\lambda_i(x+y))-\Re( \omega_{\lambda_i}\lambda_i(x)) | = O(\rho')
\end{equation*}
if $y \in B_{\rho'}''$.  Since $\mu$ is positive definite we see that $\Lambda$ is $K$-dissociated w.r.t. $\mu$ with $K=\log\int{p_{1,\Lambda}d\mu}$.  Thus, integrating the above, we get that
\begin{equation*}
\int{|\tau_y(p_{\omega,\Lambda})-p_{\omega,\Lambda}|d\beta \ast \mu}  \leq \sum_{i=1}^k{O(\rho'). \int{p_{\omega_i,\Lambda}d\beta \ast \mu}} =O(\rho' k\exp(K))
\end{equation*}
for all $y \in B_{\rho'}''$ by monotonicity of dissociativity.  Let $\rho'=\Omega(\epsilon\alpha/k)$ be such that the above error term is at most $\epsilon\alpha\exp(K)/2$ and put $B'=(B_{\rho'}''\wedge B_\rho)_\nu$ where $\nu \in [1/2,1)$ is such that $B'$ is regular.  It follows by the triangle inequality that
\begin{equation*}
\int{|p_{\omega,\Lambda} \ast \beta' - p_{\omega,\Lambda}|d\beta \ast \mu} \leq \epsilon\alpha\exp(K)/2.
\end{equation*}
On the other hand, since $A \subset B_{1-2\rho}$ we have that
\begin{equation*}
\langle 1_A, p_{\omega,\Lambda}\rangle_{L^2(\beta)} =\langle 1_A, p_{\omega,\Lambda}\rangle_{L^2(\beta \ast \mu)},
\end{equation*}
and so by the triangle inequality and hypothesis we have
\begin{equation*}
\langle 1_A, p_{\omega,\Lambda} \ast \beta'\rangle_{L^2(\beta)} \geq \alpha(1+\epsilon/2)\exp(K).
\end{equation*}
However, since $A \subset B_{1-3\rho}$, we have that
\begin{equation*}
\langle 1_A, p_{\omega,\Lambda} \ast \beta'\rangle_{L^2(\beta)} =\langle 1_A\ast \beta', p_{\omega,\Lambda}\rangle_{L^2(\beta \ast \mu)}. 
\end{equation*}
Finally, by monotonicity of dissociativity again we get that
\begin{equation*}
\|1_A \ast \beta'\|_{L^\infty(\mu_G)}\exp(K) \geq \alpha(1+\epsilon/2)\exp(K),
\end{equation*}
and the result follows.
\end{proof}

\section{Roth's theorem in high rank Bohr sets}

From now on we assume that the underlying group has odd order so that we may apply Lemma \ref{lem.basicl2}.

The purpose of this section is to establish the following theorem which is the formal version of the result discussed in \S\ref{ss.eia}.  
\begin{theorem}\label{thm.energy}
Suppose that $B$ is a rank $k$ Bohr set and $A \subset B$ has density $\alpha>0$.  Then
\begin{equation*}
T(1_A,1_A,1_A) =\exp(-\widetilde{O}(k + \alpha^{-2}))\mu_G(B)^2.
\end{equation*}
\end{theorem}
We shall prove this iteratively using the following lemma.
\begin{lemma}\label{lem.energyitlem}
Suppose that $B$ is a $d$-dimensional Bohr set such that $B$ and $B_\rho$ are regular, and $A \subset B$ and $A' \subset B_\rho$ have density $\alpha>0$ and $\alpha'>0$ respectively with\footnote{This condition is unnecessary but it makes for simpler statements, particularly later on.  It is also a slight abuse; to be formally correct the reader may wish to take, for example, $2\alpha' \geq \alpha \geq \alpha'/2$.} $\alpha = \Theta(\alpha')$.  Then at least one of the following is true:
\begin{enumerate}
\item (Many three-term progressions)
\begin{equation*}
T(1_A,1_{A'},1_A) \geq \alpha^2\alpha'\mu_G(B)\mu_G(B_\rho)/2;
\end{equation*}
\item (Energy induced density increment) or there is a regular Bohr set $B'$ with
\begin{equation*}
\rk(B') \leq \rk(B) + \widetilde{O}(\alpha^{-2}) \textrm{ and }\mu_G(B') \geq \exp(-\widetilde{O}(d+\alpha^{-2}))\mu_G(B_\rho)
\end{equation*}
such that $\|1_A \ast \beta'\|_{L^\infty(\mu_G)}\geq \alpha(1+c_{\ref{lem.energyitlem}})$;
\end{enumerate}
provided $\rho \leq c_{\ref{lem.energyitlem}}'\alpha/d$.
\end{lemma}
\begin{proof}
By Lemma \ref{lem.fdic} (provided $\rho \leq c_{\ref{lem.fdic}}'\alpha/d$) we are either in the first case of the lemma or else
\begin{equation}\label{eqn.ms}
\sum_{-2\gamma \in \Spec_{c_{\ref{lem.fdic}}\alpha}(1_{A'},\beta_\rho)}{|((1_A-\alpha)1_B)^\wedge(\gamma)|^2}\geq c_{\ref{lem.fdic}}\alpha^2\mu_G(B).
\end{equation}
By (the relative version of) the Chang bound (Lemma \ref{lem.changbd}) we see that $ \Spec_{c_{\ref{lem.fdic}}\alpha}(1_{A'},\beta_\rho)$ has $(1,\beta_\rho)$-relative entropy $O(\alpha^{-2}\log \alpha'^{-1})=\widetilde{O}(\alpha^{-2})$.  Thus, by Corollary \ref{cor.dis} there is a Bohr set $B'' \leq B_\rho$ with
\begin{equation*}
\rk(B'') \leq \rk(B) + \widetilde{O}(\alpha^{-2}) \textrm{ and } \mu_G(B'') \geq \exp(-\widetilde{O}(d+\alpha^{-2}))\mu_G(B_\rho)
\end{equation*}
such that
\begin{equation*}
|1-\gamma(x)| \leq 1/2 \textrm{ for all } x \in B'' \textrm{ and } \gamma \in \Spec_{c_{\ref{lem.fdic}}\alpha}(1_{A'},\beta_\rho).
\end{equation*}
It follows by Lemma \ref{lem.basicl2} (with $\mathcal{L}:=\Spec_{c_{\ref{lem.fdic}}\alpha}(1_{A'},\beta_\rho)$ provided $\rho \leq c_{\ref{lem.basicl2}}'c_{\ref{lem.fdic}}\alpha/d$) that we are in the second case of the lemma.
\end{proof}
\begin{proof}[Proof of Theorem \ref{thm.energy}]
We proceed iteratively to construct a sequence of regular Bohr sets $B^{(i)}$ of rank $k_i$ and dimension $d_i=O(k_i)$.  We write 
\begin{equation*}
\alpha_i:=\|1_A \ast \beta^{(i)}\|_{L^\infty(\mu_G)},
\end{equation*}
put $B^{(0)}:=B$ and suppose that we have defined $B^{(i)}$.  By regularity we have that
\begin{equation*}
(1_A \ast \beta^{(i)}_{\rho_i}+1_A \ast \beta^{(i)}_{\rho_i\rho'_i} )\ast  \beta^{(i)}(x_i) = 2\alpha_i + O(\rho_i k_i)
\end{equation*}
for some $x_i$.  Pick $\rho_i = \Omega(\alpha_i/k_i)$ and then $\rho'_i = \Omega(\alpha_i/k_i)$ such that $B^{(i)}_{\rho_i}$ and $B^{(i)}_{\rho'_i\rho_i}$ are regular,
\begin{equation*}
(1_A \ast \beta^{(i)}_{\rho}+1_A \ast  \beta^{(i)}_{\rho_i\rho'_i} \ast  \beta^{(i)}(x_i) \geq 2\alpha_i(1-c_{\ref{lem.energyitlem}}/4),
\end{equation*}
and $\rho'_i \leq c_{\ref{lem.energyitlem}}'\alpha_i/d_i$.  Thus there is some $x_i'$ such that
\begin{equation*}
1_A \ast \beta^{(i)}_{\rho}(x_i')+1_A \ast  \beta^{(i)}_{\rho_i\rho'_i}(x_i')\geq 2\alpha_i(1-c_{\ref{lem.energyitlem}}/4).
\end{equation*}
If one of the two terms is at least $\alpha_i(1+ c_{\ref{lem.energyitlem}}/4))$ then set $B^{(i+1)}$ to be $B^{(i)}_{\rho_i}$ or $B^{(i)}_{\rho_i'\rho_i}$ respectively.  Otherwise we see that $A_i:=(x_i'-A)\cap B^{(i)}_{\rho_i}$ has
\begin{equation*}
\beta^{(i)}_{\rho_i}(A_i) \geq \alpha_i(1-3c_{\ref{lem.energyitlem}}/4),
\end{equation*}
and $A_i':= (x_i' - A) \cap B^{(i)}_{\rho_i'\rho_i}$ has
\begin{equation*}
 \beta^{(i)}_{\rho_i'\rho_i}(A_i') \geq \alpha_i(1-3c_{\ref{lem.energyitlem}}/4).
\end{equation*}
By Lemma \ref{lem.energyitlem} we see that either we are done in that
\begin{equation}\label{eqn.jqz}
T(1_A,1_A,1_A) =\Omega(\alpha_i^3\mu_G(B^{(i)}_{\rho_i})\mu_G(B^{(i)}_{\rho_i'\rho_i}))
\end{equation}
or there is a regular Bohr set $B^{(i+1)}$ such that
\begin{equation*}
k_{i+1} \leq k_i + \widetilde{O}(\alpha_i^{-2}),
\end{equation*}
\begin{equation*}
\mu_G(B^{(i+1)}) \geq \exp(-\widetilde{O}(k_i+\alpha_i^{-2}))\mu_G(B^{(i)}_{\rho_i'\rho_i}) \geq \exp(-\widetilde{O}(k_i+\alpha_i^{-2}))\mu_G(B^{(i)}),
\end{equation*}
and $\alpha_{i+1} \geq \alpha_i(1+c_{\ref{lem.energyitlem}}/4)$.  Thus, if the iteration proceeds to the next step, then however it did so we have $\alpha_{i+1} \geq \alpha_i(1+\Omega(1))$.  This cannot happen for more than $O(\log 2\alpha^{-1})$ step, whence we are in (\ref{eqn.jqz}) at some step $i_0=O(\log 2\alpha^{-1})$, where we shall have
\begin{equation*}
k_{i_0} = k+\widetilde{O}(\alpha^{-2}) \textrm{ and } \mu_G(B^{(i_0)}) \geq  \exp(-\widetilde{O}(k+\alpha^{-2}))\mu_G(B),
\end{equation*}
from which the result follows.
\end{proof}

\section{The main result}

In this section we prove the following theorem which is the main result of the paper.
\begin{theorem}\label{thm.keytheorem}
Suppose that $A$ has density $\alpha>0$.  Then
\begin{equation*}
T(1_A,1_A,1_A) = \exp(-\widetilde{O}(\alpha^{-4/3})).
\end{equation*}
\end{theorem}
To see how Theorem \ref{thm.roth} follows simply note that if $A \subset \{1,\dots,N\}$ has density $\alpha$ and no non-trivial three-term progressions then it can be embedded in $\Z/ (4N+1)\Z$ with density $\Omega(\alpha)$ such that it also has no non-trivial three-term progressions there, but then
\begin{equation*}
(4N+1) \geq (4N+1)^2\exp(-\widetilde{O}(\alpha^{-4/3})), 
\end{equation*}
which can be rearranged to give the desired bound.

As before we proceed iteratively; the following is the driver.
\begin{lemma}\label{lem.mainitlem}
Suppose that $B$ is a $d$-dimensional Bohr set such that $B$ and $B_\rho$ are regular, and $A \subset B$ and $A' \subset B_\rho$ have density $\alpha>0$ and $\alpha'>0$ respectively with $\alpha = \Theta (\alpha')$.  Then at least one of the following is true:
\begin{enumerate}
\item (Many three-term progressions)
\begin{equation*}
T(1_A,1_{A'},1_A) \geq \alpha^2\alpha'\mu_G(B)\mu_G(B_\rho)/2;
\end{equation*}
\item (Energy induced density increment from large coefficients) or there is a regular Bohr set $B'$ with
\begin{equation*}
\rk(B') \leq \rk(B) + \widetilde{O}(\alpha^{-1})\textrm{ and } \mu_G(B') \geq \exp(-\widetilde{O}(d+\alpha^{-1}))\mu_G(B_\rho)
\end{equation*}
such that $\|1_A \ast \beta'\|_{L^\infty(\mu_G)}\geq \alpha(1+c_{\ref{lem.mainitlem}})$;
\item (Terminal density increment) or there is a regular Bohr set $B'$ with
\begin{equation*}
\rk(B') \leq \rk(B) + \widetilde{O}(\alpha^{-4/3})\textrm{ and } \mu_G(B') \geq \exp(-\widetilde{O}(d+\alpha^{-4/3}))\mu_G(B_\rho)
\end{equation*}
such that $\|1_A \ast \beta'\|_{L^\infty(\mu_G)}=\widetilde{\Omega}(\alpha^{2/3})$;
\item (Density increment from many independent small coefficients) or there is a regular Bohr set $B'$ with
\begin{equation*}
\rk(B') \leq \rk(B) + \widetilde{O}(\alpha^{-2/3})\textrm{ and } \mu_G(B') \geq \exp(-\widetilde{O}(d+\alpha^{-2/3}))\mu_G(B_\rho)
\end{equation*}
such that $\|1_{A'} \ast \beta'\|_{L^\infty(\mu_G)}\geq \alpha'(1+c_{\ref{lem.mainitlem}}\alpha^{1/3})$;
\end{enumerate}
provided $\rho \leq c_{\ref{lem.mainitlem}}'\alpha/d$.
\end{lemma}
\begin{proof}
The argument is not difficult although it involves a number of somewhat technical details.  To start with (for the purpose of applying Lemma \ref{lem.ctsriesz} later) we shall adjust the set $A'$ slightly.  Let $\rho'=c\alpha'^2/d$ for some absolute $c>0$ such that $B_{\rho\rho'}$ is regular and
\begin{equation*}
\mu_G(B_\rho \setminus B_{\rho(1-3\rho')}) \leq \alpha'^2/128
\end{equation*}
(possible by regularity of $B_\rho$), and put $A'':=A' \cap B_{\rho(1-3\rho')}$ so that
\begin{equation*}
A'' \subset B_{\rho(1-3\rho')} \textrm{ and } \beta_\rho(A'') \geq \alpha'(1-\alpha'/128).
\end{equation*}
By Lemma \ref{lem.fdic} applied to $A$ and $A''$ (provided $\rho \leq c_{\ref{lem.fdic}}'\alpha/d$) we are either in the first case of the lemma or else
\begin{equation}\label{eqn.ms2}
\sum_{-2\gamma \in \Spec_{c_{\ref{lem.fdic}}\alpha}(1_{A''},\beta_\rho)}{|((1_A-\alpha)1_B)^\wedge(\gamma)|^2|(1_{A''}d\beta_\rho)^\wedge(-2\gamma)|}\geq c_{\ref{lem.fdic}}\alpha^2\alpha'\mu_G(B).
\end{equation}
We assume that we are not in the first case and proceed with analysis of $\Spec_{c_{\ref{lem.fdic}}\alpha}(1_{A''},\beta_\rho)$ depending on where the mass in (\ref{eqn.ms2}) concentrates.  Let $\epsilon:=c\alpha^{1/2}$
where $c$ is an absolute constant such that
\begin{equation}\label{eqn.range}
C_{\ref{lem.newinc}}\log 2c_{\ref{lem.fdic}}^{-1}\tau^{-1} \leq c_{\ref{lem.newinc}}\tau^{-2/3} \textrm { whenever } \tau \leq c.
\end{equation}
First we split into two cases.
\begin{case*}
At least half the mass in (\ref{eqn.ms2}) is supported by those $\gamma$ with $-2\gamma \in \Spec_{\epsilon}(1_{A''},\beta_\rho)$.
\end{case*}
\begin{proof}[Analysis of case] We proceed as in the proof of Lemma \ref{lem.energyitlem}: by (the relative version of) the Chang bound (Lemma \ref{lem.changbd}) we see $\Spec_{\epsilon}(1_{A''},\beta_\rho)$ has $(1,\beta_\rho)$-relative entropy $\widetilde{O}(\alpha^{-1})$.  Thus, by Corollary \ref{cor.dis} there is a Bohr set $B''\leq B_\rho$ with
\begin{equation*}
\rk(B'') \leq \rk(B) + \widetilde{O}(\alpha^{-1}) \textrm{ and } \mu_G(B'') \geq \exp(-\widetilde{O}(d+\alpha^{-1}))\mu_G(B_\rho)
\end{equation*}
such that 
\begin{equation*}
|1-\gamma(x)| \leq 1/2 \textrm{ for all }x \in B'' \textrm{ and }\gamma \in \Spec_{\epsilon}(1_{A''},\beta_\rho).
\end{equation*}
It follows by Lemma \ref{lem.basicl2} (provided $\rho \leq c_{\ref{lem.basicl2}}'c_{\ref{lem.fdic}}\alpha/2d$) that we are in the second case of the lemma.
\end{proof}
\begin{case*}
At least half the mass in (\ref{eqn.ms2}) is supported by those $\gamma$ with $-2\gamma \not \in \Spec_{\epsilon}(1_{A''},\beta_\rho)$.
\end{case*}
\begin{proof}[Analysis of case]
By averaging there is some $\tau \in [c_{\ref{lem.fdic}}\alpha,\epsilon]$ such that the set of $\gamma$ such that $-2\gamma$ is a member of
\begin{equation*}
\mathcal{V}:=\{\gamma: 2\tau \alpha' >|(1_{A''}d\beta_\rho)^\wedge(\gamma)| \geq \tau \alpha'  \}
\end{equation*}
supports at least $\Omega(\log^{-1} 2\alpha^{-1})$ of the mass in (\ref{eqn.ms2}), whence
\begin{equation*}
\sum_{-2\gamma \in \mathcal{V}}{|((1_A-\alpha)1_B)^\wedge(\gamma)|^2}=\widetilde{\Omega}(\tau^{-1}\alpha^2)\mu_G(B).
\end{equation*}
Note that by definition of $A''$ we have that
\begin{equation*}
|(1_{A''}d(\beta_\rho \ast \beta_{\rho\rho'}))^\wedge(\gamma)| = |(1_{A''}d\beta_\rho)^\wedge(\gamma)|,
\end{equation*}
whence $\mathcal{V} \subset \Spec_{\tau}(1_{A''},\beta_\rho \ast \beta_{\rho\rho'})$.  However, the Bessel bound (Lemma \ref{lem.apbess}) tells us that $\Spec_{\tau}(1_{A''},\beta_\rho\ast \beta_{\rho\rho'})$ has $(1,\beta_\rho \ast \beta_{\rho\rho'})$-relative size $O(\alpha^{-3})$ and so by monotonicity $\mathcal{V}$ has $(c_{\ref{lem.newinc}}'c_{\ref{lem.fdic}}\alpha,\beta_{\rho\rho'})$-relative size $O(\alpha^{-3})$.

Now, apply Lemma \ref{lem.majortho} to $\mathcal{V}$  and the Bohr set $B_{\rho\rho'}$ to get a set $\Lambda$ of $k$ characters, $(c_{\ref{lem.newinc}}'c_{\ref{lem.fdic}}\alpha)$-orthogonal w.r.t. $\beta_{\rho\rho'}$ (and hence $\mu:=\beta_{\rho\rho'} \ast \beta_{\rho\rho'}$), and an order-preserving map $\phi:\mathcal{P}(\Lambda) \rightarrow \mathcal{P}(\mathcal{V})$ such that $\phi(\Lambda)=\mathcal{V}$ and for all $\Lambda' \subset \Lambda$ and $\gamma \in \phi(\Lambda' )$ we have
\begin{equation}\label{eqn.bs}
|1-\gamma(x)| \leq 1/2 \textrm{ for all }x \in B_{\rho\rho'\rho''}\wedge B'''_\nu
\end{equation}
where $B'''$ is the Bohr set with constant width function $2$ and frequency set $\Lambda'$, and $\rho''=\Omega(\alpha/d^2\log 2\alpha^{-1})$ and $\nu=\Omega(1)$.

We define a sub-additive weight $w$ on $\mathcal{P}(\Lambda)$ by
\begin{equation*}
w(\Lambda')=\sum_{-2\gamma \in \phi(\Lambda' )}{|((1_A-\alpha)1_B)^\wedge(\gamma)|^2},
\end{equation*}
and a sequence of sets $\Lambda_0,\dots,\Lambda_r$ as follows.  We begin with $\Lambda_0:=\Lambda$, and at stage $i$ if there is some $\Lambda' \subset \Lambda_i$ with 
\begin{equation*}
|\Lambda'| \geq |\Lambda_i|/2 \textrm{ and }w(\Lambda') \leq w(\Lambda)/\log 2k
\end{equation*}
then we put $\Lambda_{i+1}:=\Lambda_i \setminus \Lambda'$.  Since $|\Lambda_{i+1}| \leq |\Lambda_i|/2$ we see that the iteration terminates after at most $\log_2 k$ steps.  At this point we have a set $\Lambda_i$ such that every subset of size at least $|\Lambda_i|/2$ has weight at least $w(\Lambda_i)/\log 2k$.  On the other hand
\begin{equation*}
w(\Lambda_i) \geq \left(1-\frac{1}{\log 2k}\right)^{\lceil \log_2k\rceil} w(\Lambda) = \widetilde{\Omega}(\tau^{-1}\alpha^{2})\mu_G(B).
\end{equation*}
Once again we split into two cases.
\begin{case*}
$\Lambda_i$ is small: $|\Lambda_i| \leq \tau^{-1}\alpha^{-2/3}$.
\end{case*}
\begin{proof}[Analysis of case]
By averaging there is a subset $\Lambda' \subset \Lambda_i$ of size at most $\alpha^{-1}$ such that
\begin{equation*}
w(\Lambda') \geq \frac{\lfloor \alpha^{-1}\rfloor}{|\Lambda_i|}w(\Lambda_i) =\widetilde{\Omega}(\alpha^{5/3})\mu_G(B).
\end{equation*} 
It follows from (\ref{eqn.bs}) that there is a Bohr set $B''\leq B$ with
\begin{equation*}
\rk(B'') \leq \rk(B) + \alpha^{-1} \textrm{ and } \mu_G(B'') \geq \exp(-\widetilde{O}(d+\alpha^{-1}))\mu_G(B_\rho)
\end{equation*}
such that
\begin{equation*}
|1-\gamma(x)| \leq 1/2 \textrm{ for all } x \in B'' \textrm{ and }\gamma \in \phi(\Lambda').
\end{equation*}
We apply Lemma \ref{lem.basicl2} (provided $\rho \leq c'\alpha/d$ for some absolute $c'$) and are certainly in the third case of the lemma.
\end{proof}
\begin{case*}
$\Lambda_i$ is large: $|\Lambda_i| \geq \tau^{-1}\alpha^{-2/3}$.
\end{case*}
\begin{proof}[Analysis of case]
We apply Lemma \ref{lem.newinc} with the set $A''$, measures $\beta_\rho$ and $\mu$, and parameter $m=\alpha^{-4/3}$ to get two possibilites. 
\begin{case*}
There is a set $\Lambda' \subset \Lambda_i$ with $|\Lambda'|\geq |\Lambda_i|/2$ such that $\Lambda'$ has $(1,\mu)$-relative entropy $\widetilde{O}(\alpha^{-4/3})$.
\end{case*}
\begin{proof}[Analysis of case]
By monotonicity $\Lambda'$ has $(1,\beta_{\rho\rho'})$-relative entropy $\widetilde{O}(\alpha^{-4/3})$, and so by Lemma \ref{lem.majdissoc} applied to $\Lambda'$ and the Bohr set $B_{\rho\rho'}$ that there is a $\rho'''=(\alpha/2d)^{O(1)}$ and $\nu' = \Omega(1/k)$ such that
\begin{equation*}
|1-\gamma(x)| \leq \nu'' \textrm{ for all }x \in (B_{\rho\rho'\rho'''}\wedge B_{\nu'}'''')_\nu'' \textrm{ and } \gamma \in \Lambda'
\end{equation*}
where $B''''$ is a Bohr set with constant width function $2$ and a frequency set of size $\widetilde{O}(\alpha^{-4/3})$.  It follows that $B''' \leq B_{\rho\rho'\rho'''}\wedge B_{\nu'}''''$, and inserting this in (\ref{eqn.bs}) we get that for all $\gamma \in \phi(\Lambda' )$,
\begin{equation*}
|1-\gamma(x)| \leq 1/2 \textrm{ for all } x \in B_{\rho\rho'\min\{\rho'',\rho'''\nu\}}\wedge B_{\nu'\nu}.
\end{equation*}
We conclude that there is a Bohr set $B''\leq B_\rho$ with
\begin{equation*}
\rk(B'') \leq \rk(B) + \widetilde{O}(\alpha^{-4/3}) \textrm{ and } \mu_G(B'') \geq \exp(-\widetilde{O}(d+\alpha^{-4/3}))\mu_G(B_\rho)
\end{equation*}
such that $|1-\gamma(x)| \leq 1/2$ for all $x \in B''$ and $\gamma \in \phi(\Lambda' )$.  On the other hand
\begin{equation*}
w(\Lambda') \geq \frac{w(\Lambda_i)}{\log 2k}= \widetilde{\Omega}(\tau^{-1}\alpha^2)\mu_G(B) = \widetilde{\Omega}(\alpha^{3/2})\mu_G(B).
\end{equation*}
 It follows from Lemma \ref{lem.basicl2} that we are in the third case of the lemma provided $\rho \leq c''\alpha/d$ for some absolute $c''$.
\end{proof}
\begin{case*}
There is a set of characters $\Lambda'''$ of $O(\sqrt{m})$ characters and a function $\omega':\Lambda''' \rightarrow D$ such that
\begin{equation*}
\langle 1_{A''},p_{\omega',\Lambda'''}\rangle_{L^2(\beta_\rho)}\geq \alpha'(1+\tau \sqrt{m}/4)\int{p_{1,\Lambda'''}d\mu}.
\end{equation*}
\end{case*}
\begin{proof}[Analysis of case]
Apply Lemma \ref{lem.ctsriesz} and we are in the final case of the lemma.
\end{proof}
\end{proof}
\end{proof}
The proof is complete.
\end{proof}
\begin{proof}[Proof of Theorem \ref{thm.keytheorem}]
We proceed as before to construct a sequence of regular Bohr sets $B^{(i)}$ of rank $k_i$ and dimension $d_i=O(k_i)$.  We write 
\begin{equation*}
\alpha_i:=\|1_A \ast \beta^{(i)}\|_{L^\infty(\mu_G)},
\end{equation*}
put $B^{(0)}:=G$ and suppose that we have defined $B^{(i)}$.  By regularity we have that
\begin{equation*}
(1_A \ast \beta^{(i)}_{\rho_i}+1_A \ast  \beta^{(i)}_{\rho_i\rho'_i}) \ast  \beta^{(i)}(x_i) = 2\alpha_i + O(\rho_i k_i)
\end{equation*}
for some $x_i$.  Pick $\rho_i = \Omega(\alpha_i^{4/3}/k_i)$ and then $\rho'_i = \Omega(\alpha_i^{4/3}/k_i)$ such that $B^{(i)}_{\rho_i}$ and $B^{(i)}_{\rho'_i\rho_i}$ are regular,
\begin{equation*}
(1_A \ast \beta^{(i)}_{\rho}+1_A \ast   \beta^{(i)}_{\rho_i\rho'_i}) \ast  \beta^{(i)}(x_i) \geq 2\alpha_i(1-c_{\ref{lem.mainitlem}}\alpha_i^{1/3}/4),
\end{equation*}
and $\rho'_i \leq c_{\ref{lem.mainitlem}}'\alpha_i/d_i$.  Thus there is some $x_i'$ such that
\begin{equation*}
1_A \ast \beta^{(i)}_{\rho}(x_i')+1_A \ast  \beta^{(i)}_{\rho_i\rho'_i}(x_i')\geq 2\alpha_i(1-c_{\ref{lem.mainitlem}}\alpha_i^{1/3}/4).
\end{equation*}
If one of the two terms is at least $\alpha_i(1+ c_{\ref{lem.mainitlem}}\alpha_i^{1/3}/4))$ then set $B^{(i+1)}$ to be $B^{(i)}_{\rho_i}$ or $B^{(i)}_{\rho_i'\rho_i}$ respectively.  Otherwise we see that $A_i:=(x_i'-A)\cap B^{(i)}_{\rho_i}$ has
\begin{equation*}
\beta^{(i)}_{\rho_i}(A_i) \geq \alpha_i(1-3c_{\ref{lem.mainitlem}}\alpha_i^{1/3}/4),
\end{equation*}
and $A_i':= (x_i' - A) \cap B^{(i)}_{\rho_i'\rho_i}$ has
\begin{equation*}
\beta^{(i)}_{\rho_i'\rho_i}(A_i') \geq \alpha_i(1-3c_{\ref{lem.mainitlem}}\alpha_i^{1/3}/4).
\end{equation*}
By Lemma \ref{lem.mainitlem} we see that at least one of the following occurs:
\begin{enumerate}
\item we stop the iteration with
\begin{equation*}
T(1_A,1_A,1_A) =\Omega(\alpha_i^3\mu_G(B^{(i)}_{\rho_i})\mu_G(B^{(i)}_{\rho_i'\rho_i}));
\end{equation*}
\item there is a regular Bohr set $B'$ with $\rk(B') \leq k_i+ \widetilde{O}(\alpha_i^{-4/3})$ and
\begin{equation*}
\mu_G(B') \geq \exp(-\widetilde{O}(k_i + \alpha_i^{-4/3}))\mu_G(B^{(i)}_{\rho_i'\rho_i})\geq \exp(-\widetilde{O}(k_i+\alpha_i^{-4/3}))\mu_G(B^{(i)})
\end{equation*}
such that $\|1_A \ast \beta'\|_{L^\infty(\mu_G)} = \widetilde{\Omega}(\alpha^{2/3})$;
\item there is a regular Bohr set $B^{(i+1)}$ with $k_{i+1} \leq k_i + \widetilde{O}(\alpha_i^{-2/3})$ and,
\begin{equation*}
\mu_G(B^{(i+1)}) \geq \exp(-\widetilde{O}(k_i+\alpha_i^{-2/3}))\mu_G(B^{(i)}_{\rho_i'\rho_i}) \geq \exp(-\widetilde{O}(k_i+\alpha_i^{-2/3}))\mu_G(B^{(i)}),
\end{equation*}
such that $\alpha_{i+1} \geq \alpha_i(1+c_{\ref{lem.mainitlem}}\alpha_i^{1/3}/4)$;
\item there is a regular Bohr set $B^{(i+1)}$ with $k_{i+1} \leq k_i + \widetilde{O}(\alpha_i^{-1})$ and,
\begin{equation*}
\mu_G(B^{(i+1)}) \geq \exp(-\widetilde{O}(k_i+\alpha_i^{-1}))\mu_G(B^{(i)}_{\rho_i'\rho_i}) \geq \exp(-\widetilde{O}(k_i+\alpha_i^{-1}))\mu_G(B^{(i)}),
\end{equation*}
such that $\alpha_{i+1} \geq \alpha_i(1+c_{\ref{lem.mainitlem}}/4)$.
\end{enumerate}
In the second case we apply Theorem \ref{thm.energy} to this translate and we are done:
\begin{equation*}
T(1_A,1_A,1_A) \geq \exp(-\widetilde{O}(k_i+\alpha_i^{-4/3}))\mu_G(B^{(i)})^2.
\end{equation*}
Otherwise we can be in the third case at most $\widetilde{O}(\alpha^{-1/3})$ times and in the fourth case $O(\log 2\alpha^{-1})$ times.  It follows that the iteration must terminate at some step $i_0=\widetilde{O}(\alpha^{-1/3})$, where we shall have $k_{i_0} =\widetilde{O}(\alpha^{-1})$, and
\begin{equation*}
\mu_G(B^{(i_0)}) \geq  \exp(-\widetilde{O}(\alpha^{-4/3}))\mu_G(B),
\end{equation*}
from which the result follows.
\end{proof}

\section*{Acknowledgement}

The author should like to thank an anonymous referee for useful comments.

\bibliographystyle{alpha}

\bibliography{references}

\end{document}